\documentclass[12pt]{article}
\usepackage{latexsym}
\usepackage{amssymb}
\usepackage{amscd} 
\usepackage{amsmath} 
\usepackage{amsrefs}
\usepackage{amsthm}
\usepackage{color}
\usepackage{graphicx} 
\usepackage[hidelinks]{hyperref}
\hypersetup{
    colorlinks=true, 
    linktoc=all,     
    allcolors=black,  
}
\usepackage{mathtools}
\usepackage{caption}
\usepackage{subcaption}
\usepackage{tikz}
\usepackage{tikz-cd}
\usetikzlibrary{arrows}
\usetikzlibrary{patterns}
\usepackage{enumitem}

\normalsize

\def\wl{\par \vspace{\baselineskip}}

\def\dfrac#1#2{\lower0.15ex\hbox{\large$\frac{#1}{#2}$}}

\let\originalleft\left
\let\originalright\right
\renewcommand{\left}{\mathopen{}\mathclose\bgroup\originalleft}
\renewcommand{\right}{\aftergroup\egroup\originalright}

  \newtheorem{theorem}{Theorem}[section]
  \newtheorem{lemma}[theorem]{Lemma}
  
  \newtheorem{cor}[theorem]{Corollary}
  
  \theoremstyle{definition}

  \newtheorem{problem}[theorem]{Problem}

\begin{document}

\title{Eulerian circuits and path decompositions \\in quartic planar graphs}
\author{Jane Tan\footnote{Supported by the Australian Research Council, Discovery Project DP140101519.}\\
 \small Mathematical Sciences Institute, Australian National University\\[-0.6ex]
 \small Canberra, ACT 2601, Australia\\[-0.6ex]
 \small \texttt{jane.tan@maths.ox.ac.uk}
}
\date{}
\maketitle

\begin{abstract}
A subcycle of an Eulerian circuit is a sequence of edges that are consecutive in the circuit and form a cycle. We characterise the quartic planar graphs that admit Eulerian circuits avoiding 3-cycles and 4-cycles. From this, it follows that a quartic planar graph of order $n$ can be decomposed into $k_1+k_2+k_3+k_4$ many paths with $k_i$ copies of $P_{i+1}$, the path with $i$ edges, if and only if $k_1+2k_2+3k_3+4k_4 = 2n$. In particular, every connected quartic planar graph of even order admits a $P_5$-decomposition.
\end{abstract}

\section{Introduction}
The graphs in this paper are finite and simple. An \emph{edge-decomposition} of a graph $G$ is a collection of edge-disjoint subgraphs whose union is $G$. If each subgraph is isomorphic to $P_k$, the path on $k$ vertices, then we call this a $P_k$-decomposition. Every graph has a $P_2$-decomposition, and it is a classical result that a graph $G$ has a $P_3$-decomposition if and only if $G$ has an even number of edges. Kotzig~\cite{kotzigeven} showed that a cubic graph admits a $P_4$-decomposition if and only if it has a perfect matching (see also \cite{BF83}), and went on to ask for necessary and sufficient conditions for a $d$-regular simple graph to admit a $P_{d+1}$-decomposition when $d>3$ is odd (see~\cite{FGK10}). Bondy~\cite{bondypath} later posed a more general question: 

\begin{problem}\label{prpath}
Which simple graphs admit a $P_k$-decomposition?
\end{problem}

There is no complete answer for any $k\geq 4$. For Eulerian graphs, one way to prove the existence of a $P_k$-decomposition is to consider a strictly stronger problem which is also of independent interest. Let $C= x_1x_2\ldots x_n$ be an Eulerian circuit in a simple graph. A \emph{subcycle} of $C$ is a subpath $x_ix_{i+1}\ldots x_j$ of $C$ such that $x_i=x_j$ and the vertices $x_i, x_{i+1},\ldots, x_{j-1}$ are all distinct. We call an Eulerian circuit or trail \emph{$k$-locally self-avoiding} for some positive integer $k$ if it contains no subcycles of length at most $k$. When $k=3$, the term \emph{triangle-free} is commonly used instead. The corresponding question is then:
\begin{problem}\label{prcircuit}
Which simple graphs admit a $k$-locally self-avoiding Eulerian circuit?
\end{problem}

Adelgren~\cite{Ade95} and Heinrich, Liu and Yu~\cite[Theorem 3.1]{HLY99} independently characterised the Eulerian graphs of maximum degree 4 that admit triangle-free Eulerian circuits. Using this characterisation, Heinrich, Liu and Yu prove that a connected quartic graph on $3n$ vertices admits a $P_4$-decomposition if and only if it has a triangle-free Eulerian circuit.  In addition, Oksimets~\cite{oksphd} showed that every Eulerian graph with minimum degree at least 6 admits a triangle-free Eulerian circuit. For results on Eulerian circuits avoiding longer subcycles in concrete classes of graphs, the literature is limited to complete and complete bipartite graphs. This includes the work of Jimbo~\cite{Jim2}, Ram\'irez-Alfons\'in~\cite{ramirez}, and Oksimets~\cite{Oks05} who also gave the precise $k$ for which $P_k$-decompositions exist in these cases. 

In this paper, we characterise the connected quartic planar graphs that admit 4-locally self-avoiding Eulerian circuits, and hence show that a connected quartic planar graph has a $P_5$-decomposition if and only if it has even order. The main theorems together with proof outlines are presented in Section~\ref{sec:overview}. The induction arguments used by Adelgren and Heinrich, Liu, and Yu to obtain their characterisations do not extend directly to 4-locally self-avoiding Eulerian circuits. The main barrier is that they quickly lead to an impractical amount of casework, even when planarity is assumed. We also use induction, but it is made feasible by applying a recursive generation theorem to simplify and organise the cases.

From the edge-decomposition perspective, the existence of $P_5$-decompositions has previously been shown for all graphs with $4e$ edges and edge-connectivity at least $10^{10^{10^{14}}}$ by Thomassen~\cite{thomassen08}, and for 8-regular graphs by Botler and Talon~\cite{BT17}. For longer paths, there are further results for graphs that are already free of certain short subcycles or have high degrees such as those in~\cite{BMW15, BMOW17, FGK10}. 

Problem~\ref{prpath} is also related to the Bar\'at-Thomassen conjecture~\cite{BT06}, which was recently proved by Bensmail, Harutyunyan, Le, Merker, and Thomass\'e~\cite{BHLMT16}. The path-case of this conjecture, first verified by Botler, Mota, Oshiro, and Wakabayashi~\cite{BMOW17_B}, says that for each $k$, there is some $c_k$ such that every $c_k$-edge-connected graph $G$ for which $k$ divides $|E(G)|$ admits a $P_{k+1}$-decomposition. This is illustrated, for example, by Thomassen's theorem mentioned above. The known edge-connectivity is challenging to decrease (see~\cite{thomassen13}). However, Bensmail, Harutyunyan, Le, and Thomass\'e~\cite{BHLT17} showed that there is some trade-off possible between high edge-connectivity and high minimum degree.

There is a similar conjecture for locally self-avoiding Eulerian circuits. H\"aggkvist and Kriesell independently conjectured that for every positive integer $k$, there is an integer $d_k$ such that every Eulerian graph with minimum degree at least $d_k$ admits a $k$-locally self-avoiding Eulerian circuit. This was recently confirmed by Le~\cite{Le19}. The minimum such values $d_k$ are unknown except for $d_3$, which is 6 by an aforementioned theorem of Oksimets~\cite{oksphd}.

In terms of progress on Problem~\ref{prpath} and Problem~\ref{prcircuit}, relatively little is known for graphs that are allowed to have modest degrees and low connectivity. We are not aware of any previous results concerning $k$-locally self-avoiding Eulerian circuits in graph classes of this description beyond $k=3$.

\section{Overview and notation}\label{sec:overview}
A \emph{trail} of length $\ell$ in a simple graph $G$ is a sequence of vertices $v_0v_1\ldots v_\ell$ such that $v_{i-1}v_{i}\in E(G)$ for each $1\leq i\leq \ell$ and there are no repeated edges. A trail is \emph{closed} if $v_\ell = v_0$ in which case it is called a $\emph{circuit}$, and \emph{Eulerian} if it traverses every edge of the graph. A \emph{path} is a trail in which no vertices are repeated. As a shorthand, we call an Eulerian trail or circuit \emph{good} if it is $4$-locally self-avoiding, and define a \emph{short} cycle to be one of length 3 or 4. For other terminology, we follow~\cite{bondymurty}.

Although our graphs are undirected, it is notationally convenient to pick a direction of traversal for trails. To this end, we write a trail $X=v_0v_1\ldots v_\ell$ as $v_0X'v_\ell$ where $X'=v_1\ldots v_{\ell-1}$ to indicate that we are starting at $v_0$ and ending at $v_\ell$. The reverse path $v_\ell v_{\ell-1}\ldots v_0$ is denoted by $X^{-1} = v_\ell {X'}^{-1}v_0$. For a vertex $v_i$ in $X$, we can \emph{split} $X$ at $v_i$ into segments $A$ and $B$ by writing $X=v_0Av_iBv_\ell$ so that $A=v_1\ldots v_{i-1}$ and $B=v_{i+1}\ldots v_\ell$. Although formally segments are just subtrails, it is useful to think of them as including the edges on either end. In this way, it is possible for a segment to have no vertices. For instance, if $X=v_0v_1$, then we could still write $X=v_0Av_1$ and in this case call $A$ an edge. We use $x_{-1}$ and $x_{-2}$ to denote the last and second last vertices in $X$ respectively. 

The main result of this paper is:
\begin{theorem}\label{thm:connected}
A quartic planar graph has a 4-locally self-avoiding Eulerian circuit if and only if it is connected and does not contain the graph $F_6$ (the complement of $P_2 \cup P_4$, shown in Figure~\ref{fig:obstructions}) as a subgraph.
\end{theorem}
\begin{figure}[ht]
	\small
	\centering
	\unitlength=1cm
	 \scalebox{0.9}{
\begin{picture}(9,3.6)(0,0)
  \put(0,0.7){\includegraphics[scale=0.5]{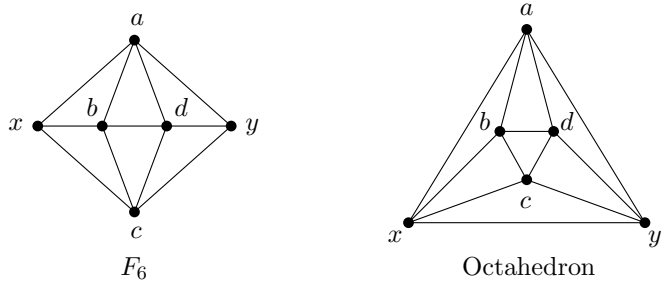}}
  \put(6.35,0){Octahedron}
  \put(7.2,3.85){$a$}
  \put(5.25,0.5){$x$}
  \put(9.1,0.5){$y$}
  \put(6.6,2.15){$b$}
  \put(7.8,2.15){$d$}
  \put(7.2,1.0){$c$}
  \put(1.3,0){$F_6$}
  \put(-0.35,2.15){$x$}
  \put(1.45,3.7){$a$}
  \put(0.8,2.35){$b$}
  \put(1.45,0.55){$c$}
  \put(2.1,2.35){$d$}
  \put(3.15,2.15){$y$}
\end{picture}}
\caption{The graph $F_6$ is an obstruction to 4-locally self-avoiding Eulerian circuits. The octahedron is the only quartic graph which contains $F_6$ as a non-induced subgraph.}
\label{fig:obstructions}
\end{figure}

The proof will be presented over the next three sections. We first show that all 3-connected quartic planar graphs except the octahedron admit 4-locally self-avoiding Eulerian circuits (Theorem~\ref{thm:no34_original}) in Section~\ref{sec:3con}. The proof of Theorem~\ref{thm:no34_original} is by induction, using a recursive graph generation theorem. Broadly, such theorems give a method of constructing every graph in a particular class from a set of starting graphs by a sequence of local expansion operations. The content of the proof is to show that each expansion operation preserves the property of admitting a good Eulerian circuit.

Section~\ref{sec:weaken} is devoted to removing the 3-connectedness condition. For this, it is helpful to classify small minimal cuts in quartic planar graphs based on the number of edges going from the cut into each ``side". Note that \emph{cut} should be read as vertex-cut by default, and is used in contexts when edge-cuts are not considered. More precisely, given a minimal vertex-cut $C$ in a quartic graph $G$, the \emph{sides associated with $C$} are two subgraphs of $G$, say $A$ and $B$, such that $V(A) \cap V(B) = C$, $A\cup B = G$, and every edge in $G$ either has both ends in $A$ or both ends in $B$. Similarly, if $C$ is a minimal edge-cut, then its sides are subgraphs $A$ and $B$ such that $\{E(A), E(B), C\}$ is a partition of $E(G)$, and every edge not in $C$ either has both ends in $A$ or both ends in $B$. 

For parity reasons, the sides of a cutvertex must each contain two neighbours, and similarly there are four possibilities for minimal 2-vertex-cuts. These are all shown in Figure~\ref{fig:allcuts}, in which only the vertices in the cut and their incident half-edges are drawn. We say that a vertex-cut $C$ is of type (a) through (e) if its sides contain the number of half-edges incident to each vertex of $C$ as the representatives shown in the figure. The sides of a cut are not necessarily unique, so it is possible for a cut to be of more than one type although this is not a concern in our proofs.

\begin{figure}[ht]
	\small
	\centering
	\unitlength=1cm
	 \scalebox{0.9}{
\begin{picture}(13,2.5)(0,0)
  \put(0,0.5){\includegraphics[scale=0.5]{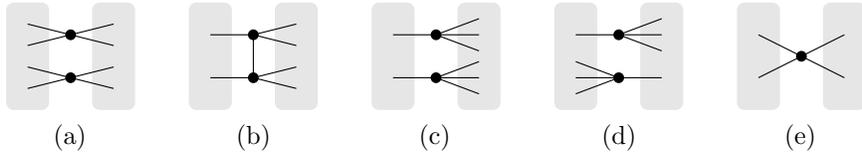}}
  \put(0.7,0){(a)}
  \put(3.4,0){(b)}
  \put(6.1,0){(c)}
  \put(8.8,0){(d)}
  \put(11.5,0){(e)}
\end{picture}}
\caption{All 1-vertex-cut and 2-vertex-cut types in quartic planar graphs. Shaded regions represent the sides excluding the cut itself.}
\label{fig:allcuts}
\end{figure}

To extend Theorem~\ref{thm:no34_original}, the first step is to relax 3-connectedness to 3-edge-connectedness (Section~\ref{sec:3edgecon}). Given 3-edge-connectedness by inspecting Figure~\ref{fig:allcuts} we observe that the only type of $2$-vertex-cut that does not imply the existence of a 2-edge-cut is type (a). Thus, we use induction on the number of 2-vertex-cuts of type (a). We then move from 3-edge-connectedness to 2-vertex-connectedness (Section~\ref{sec:2con}) by introducing the remaining types of 2-vertex-cuts, namely (b), (c) and (d). We have already observed that each of these implies the existence of a 2-edge-cut, and equivalently this can be phrased as induction on the number of 2-edge-cuts. Any graph that contains $F_6$ as an induced subgraph must have a 2-edge-cut, so it is in this step that we detect the obstruction. Finally, the main theorem is proved by induction on the number of cutvertices, which are all of type (e) (Section~\ref{sec:1con}).

The corresponding path-decomposition corollary is derived in Section~\ref{sec:cutitup}. Namely,
\begin{theorem}\label{thm:p5decomp}
A connected quartic planar graph has a $P_5$-decomposition if and only if it has even order.
\end{theorem}
It is unsurprising that $F_6$ is no longer an obstruction since it can be decomposed into two copies of $P_5$ and one copy of $P_4$. In fact, we prove the stronger statement that each quartic planar graph of order $n$ has a decomposition into $k_1+k_2+k_3+k_4$ many paths with $k_i$ copies of $P_{i+1}$, the path of length $i$, if and only if $k_1+2k_2+3k_3+4k_4 = 2n$. That is, the clearly necessary divisibility condition is also sufficient.

\section{The 3-connected case}\label{sec:3con}
The goal of this section is to prove the following:
\begin{theorem}\label{thm:no34_original}
Every 3-connected quartic planar graph except the octahedron has a $4$-locally self-avoiding Eulerian circuit. 
\end{theorem}

The proof is by induction, facilitated by the following recursive generation theorem for the class of 3-connected quartic plane graphs. It is obtained by taking the dual of Theorem~3 in \cite{quads} due to Brinkmann, Greenberg, Greenhill, McKay, Thomas and Wollan, together with the expansion operation $\overline{\phi}_{B}$ from Broersma, Duivestijn and G\"obel's earlier generation theorem for the same class (see \cite{BDG93}, Theorem 1). The reducible configurations provided by this extra operation will be an advantage to us in the proof of Lemma~\ref{pegging}.

\begin{theorem}\label{hybridgen}
The class of 3-connected quartic plane graphs can be generated from the antiprisms by pegging, 4-cycle addition, and 3-cycle slides depicted in Figure~\ref{fig:expansions}.
\end{theorem}
\begin{figure}[h]
	\small
	\centering
\includegraphics[scale=0.55]{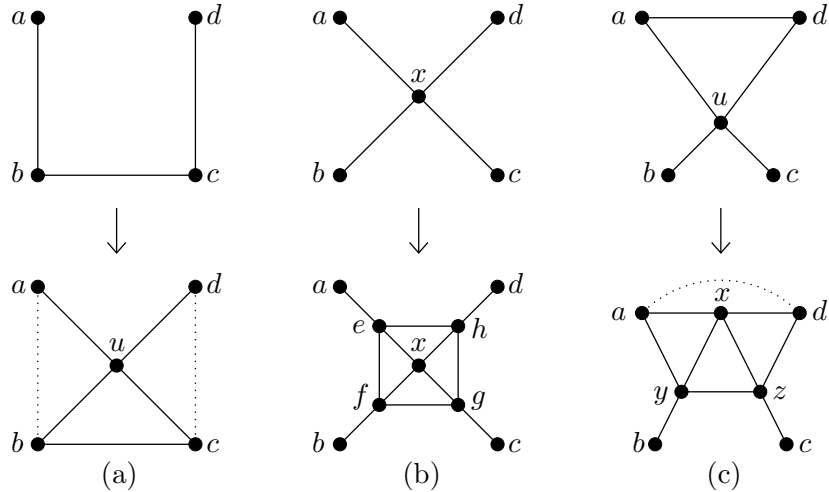}
\put(-301,162){$a$}
\put(-301,102){$b$}
\put(-301,60){$a$}
\put(-301,0){$b$}
\put(-227,162){$d$}
\put(-227,102){$c$}
\put(-227,60){$d$}
\put(-227,0){$c$}
\put(-264,38){$u$}
\put(-267,-12){(a)}
\put(-187,162){$a$}
\put(-187,102){$b$}
\put(-113,162){$d$}
\put(-113,102){$c$}
\put(-187,60){$a$}
\put(-187,0){$b$}
\put(-113,60){$d$}
\put(-113,0){$c$}
\put(-150,38){$x$}
\put(-150,140){$x$}
\put(-172,45){$e$}
\put(-172,18){$f$}
\put(-127,18){$g$}
\put(-127,43){$h$}
\put(-153,-12){(b)}
\put(-74,162){$a$}
\put(-62,102){$b$}
\put(2,162){$d$}
\put(-8,102){$c$}
\put(-36,131){$u$}
\put(-74,50){$a$}
\put(-66,0){$b$}
\put(2,50){$d$}
\put(-3,0){$c$}
\put(-35,57){$x$}
\put(-58,20){$y$}
\put(-13,20){$z$}
\put(-38,-12){(c)}
\vspace{-2mm}
\caption{From left to right, the expansion operations of pegging, 4-cycle addition and 3-cycle slide are depicted. Dotted lines indicate edges in the complement.}
\label{fig:expansions}
\end{figure}

As part of the induction argument, we show that all of the starting graphs have a 4-locally self-avoiding Eulerian circuit, and that each of these operations preserves the property of there existing such a circuit. The latter is unfortunately laborious for the pegging operation, but we will start with the other two for which it is easy. Throughout, we freely use the fact that 3-connected planar graphs have a unique embedding on the sphere, and assume a fixed choice of outer face.

\begin{lemma}\label{cycleaddition}
Suppose $G$ and $H$ are 3-connected quartic planar graphs such that $G$ can be constructed from $H$ by a single 4-cycle addition. If $H$ admits a 4-locally self-avoiding Eulerian circuit, then $G$ does as well.
\end{lemma}
\begin{proof}
Taking a 4-locally self-avoiding Eulerian circuit in $H$, we will show that it can be extended to include the new edges created in the expansion whilst still avoiding 3- and 4-cycles. The edge $ax$ is directly followed by one of $bx$, $cx$ and $dx$ in the circuit. First suppose that $axb$ is a subpath. This forces $cxd$ (or equivalently, $dxc$) to also be a subpath of the circuit. To extend, replace $axb$ by the path $aexhgfb$, and $cxd$ by $cgxfehd$ as indicated in the Figure~\ref{fig:cycleaddn}(a). The resulting circuit is certainly Eulerian, and we have not created any short subcycles within either of the paths. There is also no chance of creating short subcycles nearby; if we identify $ax$, $bx$, $cx$ and $dx$ in $H$ with $ae$, $bf$, $cg$ and $dh$ in $G$ respectively, then it can be observed that any subcycle involving $x$ increases in length when we extend the circuit. The case that we have $axd$ is symmetric, and an appropriate extension can be obtained by rotating the figure.

If instead we have $axc$ in the original circuit, then we must also have $bxd$. We replace these paths by $aefxhgc$ and $bfgxehd$ respectively (Figure~\ref{fig:cycleaddn}(b)). Again, the resulting circuit is Eulerian, and by the same reasoning as before it does not have any short subcycles. 
\end{proof}

\begin{figure}[h]
	\small
	\centering
\includegraphics[scale=0.55]{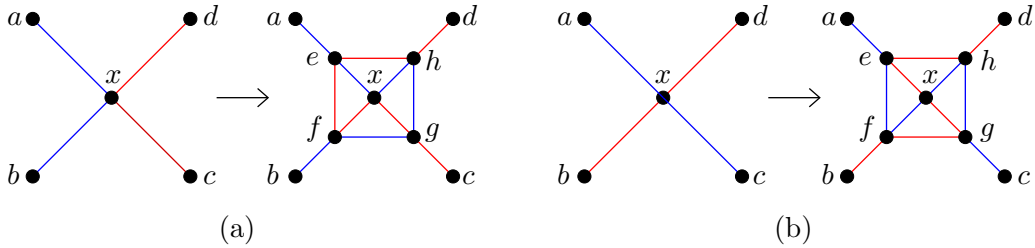}
\put(-380,60){$a$}
\put(-380,0){$b$}
\put(-306,60){$d$}
\put(-306,0){$c$}
\put(-343,38){$x$}
\put(-172,60){$a$}
\put(-172,0){$b$}
\put(-98,60){$d$}
\put(-98,0){$c$}
\put(-135,38){$x$}
\put(-282,60){$a$}
\put(-282,0){$b$}
\put(-208,60){$d$}
\put(-208,0){$c$}
\put(-244,38){$x$}
\put(-267,45){$e$}
\put(-267,18){$f$}
\put(-222,18){$g$}
\put(-222,43){$h$}
\put(-72,60){$a$}
\put(-72,0){$b$}
\put(2,60){$d$}
\put(2,0){$c$}
\put(-34,38){$x$}
\put(-58,45){$e$}
\put(-58,18){$f$}
\put(-12,18){$g$}
\put(-12,43){$h$}
\put(-300,-20){(a)}
\put(-90,-20){(b)}
\caption{Extension of an Eulerian circuit after a 4-cycle addition.}
\label{fig:cycleaddn}
\end{figure}

\begin{lemma}\label{cycleslide}
Suppose $G$ and $H$ are 3-connected quartic planar graphs such that $G$ can be constructed from $H$ by a single 3-cycle slide. If $H$ admits a 4-locally self-avoiding Eulerian circuit, then $G$ does as well.
\end{lemma}
\begin{proof}
We begin by taking a 4-locally self-avoiding Eulerian circuit in $H$. At vertex $u$, the circuit can either proceed straight ahead or turn. The former occurs when $auc$ is in the original circuit. To obtain an Eulerian circuit in $G$, we make the following substitutions; $auc$ becomes $ayzc$, $bud$ becomes $byxzd$, and $ad$ becomes $axd$.The only way a short subcycle may occur is if $adub$ was a subpath of the original circuit, causing the edges of the triangle $xzd$ to be traversed consecutively. In this case, we use the alternative substitutions replacing $auc$ with $ayxzc$, $bud$ with $byzd$, and $ad$ with $axd$. These are illustrated in Figure~\ref{fig:cycleslide}(a). With reference to the colours in the figure, we would have another problematic triangle $axy$ if $ad$ were to occur consecutive to the red edges in the circuit. However, this is impossible since if $ad$ were adjacent to both the red and green paths, the triangle $uadu$ would be a subcycle of the original circuit which was assumed to be 4-locally self-avoiding.

In the second case, the original circuit contains $aub$ and we obtain an Eulerian circuit in $G$ by one of the following substitutions; either $aub$ becomes $ayb$, $cud$ becomes $czyxd$, and $ad$ becomes $axzd$, or we replace $aub$ by $axzyb$, $cud$ by $czd$, and $ad$ by $ayxd$. These are shown in Figure~\ref{fig:cycleslide}(b). Neither of these can create 4-cycles. The former yields a circuit that will only have a 3-cycle if $aduc$ is a subpath, whilst the latter can only produce a 3-cycle if $buad$ is a subpath. We have already noted that these cannot be true simultaneously, so one of the resulting Eulerian circuits is free of short subcycles.
\end{proof}

\begin{figure}[h]
	\centering
 \scalebox{0.99}{
	\small
\includegraphics[scale=0.55]{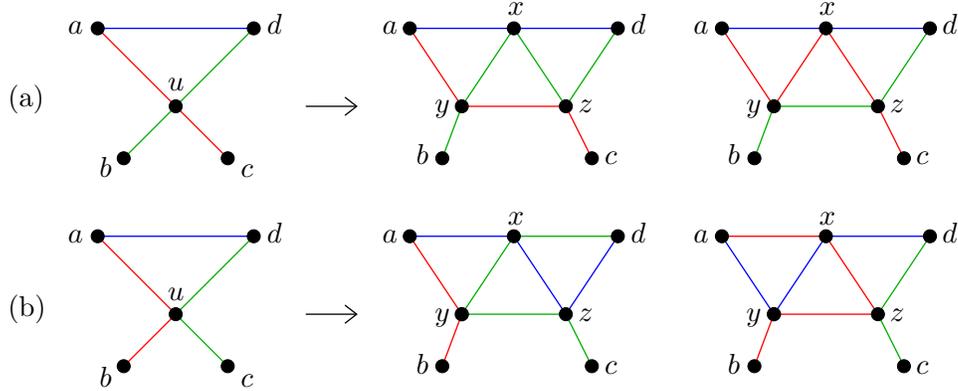}
\put(-332,130){$a$}
\put(-320,75){$b$}
\put(-256,130){$d$}
\put(-266,75){$c$}
\put(-294,108){$u$}
\put(-332,50){$a$}
\put(-320,-5){$b$}
\put(-256,50){$d$}
\put(-266,-5){$c$}
\put(-294,28){$u$}
\put(-212,130){$a$}
\put(-199,80){$b$}
\put(-117,130){$d$}
\put(-127,80){$c$}
\put(-164,137){$x$}
\put(-192,100){$y$}
\put(-137,100){$z$}
\put(-212,50){$a$}
\put(-199,0){$b$}
\put(-117,50){$d$}
\put(-127,0){$c$}
\put(-164,57){$x$}
\put(-192,20){$y$}
\put(-137,20){$z$}
\put(-93,130){$a$}
\put(-80,80){$b$}
\put(2,130){$d$}
\put(-8,80){$c$}
\put(-45,137){$x$}
\put(-73,100){$y$}
\put(-18,100){$z$}
\put(-93,50){$a$}
\put(-80,0){$b$}
\put(2,50){$d$}
\put(-8,0){$c$}
\put(-45,57){$x$}
\put(-73,20){$y$}
\put(-18,20){$z$}
\put(-355,102){(a)}
\put(-355,22){(b)}}
\caption{Extension of an Eulerian circuit after a 3-cycle slide. Edges of the same colour are assumed to appear consecutively, and the blue edge $ad$ may or may not be consecutive to either the red or green edges.}
\label{fig:cycleslide}
\end{figure}

\begin{lemma}\label{lemma:pathinsquare}
Suppose we have a connected plane graph $W$ with two vertices of degree 1, one vertex of degree 2 and all other vertices have degree 4. Suppose further that the vertices of degree less than 4 are pairwise non-adjacent and lie on the outer face. Then $W$ has at least 5 vertices of degree 4.
\end{lemma}
\begin{proof}
Let $C$ be the closed walk formed by the edges between the vertices of degree 4 that share a face with the vertex of degree 2. If $C$ has five or more vertices, then we are done. Otherwise, $C$ is a simple cycle with 3 or 4 vertices, and 2 or 4 half-edges respectively for which the other endpoint has degree 4. Moreover, those half-edges point inward since they are not in $C$ and by assumption the outer face is incident to the vertex of degree 2. In the case that it is $C_3$, no chords are possible and a single vertex inside $C$ would have at most degree 3, so there must be at least two such vertices. In the case that it is $C_4$, at most one chord is possible which would account for only two of the half-edges, so there must be a vertex inside $C$.
\end{proof}

For our purposes, it would have been enough to require three vertices of degree four in the previous lemma, but five is a sharp lower bound. 

\begin{lemma}\label{pegging}
Suppose $G$ and $H$ are 3-connected quartic planar graphs such that $G$ can be constructed from $H$ by a single pegging. If $H$ admits a 4-locally self-avoiding Eulerian circuit, then $G$ does as well.
\end{lemma}
When interpreting diagrams in the following proof, the order of half-edges and dotted trails that occur consecutively in the cyclic ordering at a vertex is unimportant. A common situation that we see is that of a \emph{trapped} vertex, which refers to a degree-deficient vertex that cannot be included by extending the current circuit without either creating a 2-vertex-cut or violating planarity. Figure~\ref{fig:trapped} depicts two representative traps. Suppose we have the black subgraph on the left with cyclic orderings as shown, and all other edges are in a single segment $uEb$. Then $be_{-1}$ is outside the region bounded by $uvwcu$. In order for $E$ to contain vertex $w$, the edge $e_{1}u$ must be inside that cycle since there are no other degree-deficient vertices that act as entry points into this bounded region. However, as there is no exit point, this means that $E$ cannot be connected which is a contradiction. In the second case, we have introduced a vertex $x$. To avoid the previous contradiction, we must have the cyclic orderings of black edges and $E$-edges as drawn, but then $\{b,x\}$ is a 2-vertex-cut. In both cases, we say that $w$ is trapped.

\begin{figure}[h]
	\small
	\centering
\includegraphics[scale=0.5]{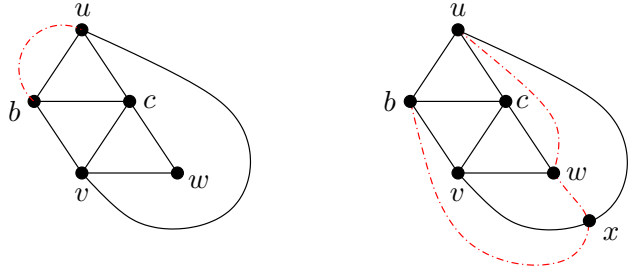}
\put(-210,95){$u$}
\put(-235,55){$b$}
\put(-184,60){$c$}
\put(-210,25){$v$}
\put(-167,30){$w$}
\put(-68,95){$u$}
\put(-93,59){$b$}
\put(-43,60){$c$}
\put(-68,25){$v$}
\put(-24,33){$w$}
\put(-10,10){$x$}
\caption{Two situations in which vertex $w$ is trapped.}
\label{fig:trapped}
\end{figure}

\begin{proof}[Proof of Lemma~\ref{pegging}]
Let the expanded graph $G$ be obtained from $H$ by pegging two edges $ab$ and $cd$ of a path $abcd$, to produce a new vertex $u$ as in Figure~\ref{fig:expansions}(a). By assumption, $H$ has a good Eulerian circuit. This induces an Eulerian circuit $C$ in $G$ given by replacing $ab$ with $aub$ and $cd$ with $cud$ appropriately oriented. 

We shall characterise the situations in which this induced circuit is not 4-locally self-avoiding. Any subcycle in the circuit that does not contain the vertex $u$ was also present in the original circuit in $H$, and hence cannot be a 3-cycle or 4-cycle. Therefore, any short subcycle must contain $u$. Moreover, since $aub$ and $cud$ are subpaths of $C$, if $u$ were to lie on two short subcycles then $C$ would have to be of the form $uPuQu$ where $uPu$ and $uQu$ have length at most four. That would imply that $G$ has at most 8 edges, which is impossible. Thus, the induced circuit contains at most one short subcycle. 

By the definition of pegging we know that $ab$ and $cd$ are non-edges. It is also straightforward to deduce that $ac$ and $bd$ are also non-edges. For example, if $bd$ were an edge then the face bounded by the walk $bduaub$ must have an even number of half-edges pointing into the face. There are already 3 such half-edges from $a$, and there must be at least one from each of $b$ and $d$ to avoid a 2-vertex-cut. This forces all six half-edges from $a$, $b$, and $d$ to point into this face, but then $c$ would be a cutvertex. This tells us the possible edges near the pegging, so the potential problem cycles are as follows:

\begin{itemize}[topsep=1pt]
\itemsep=0mm
\item If $u$ is part of a 3-cycle, then the two possibilities are $ubcu$ and $uadu$. These are symmetric, so we will only handle the former.
\item If $u$ is part of a 4-cycle, then there are four possible forms; $ubvcu$, $uavdu$, $ubvdu$, $uavcu$ where $v$ is some vertex in $G$ distinct from $a$, $b$, $c$, $d$ and $u$. The third and fourth cases are symmetric. Note also that $uavbu$ and $ucvdu$ are not possible, since in $H$ they would correspond to cycles $avba$ and $cvdc$. 
\end{itemize}
The remainder of the proof resolves each of the four distinct types by rerouting the circuit.

\begin{figure}[h]
	\small
	\centering
	\scalebox{0.9}{
\includegraphics[scale=0.65]{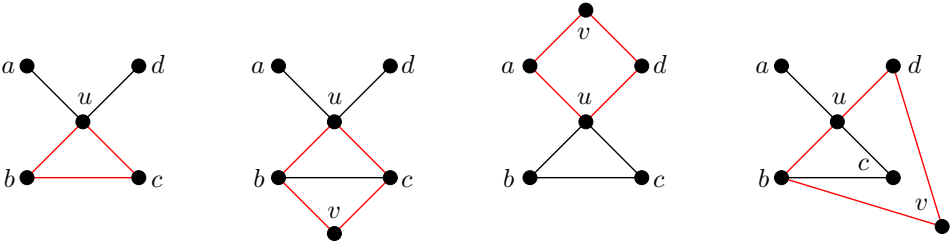}
\put(-50,58){$u$}
\put(-82,71){$a$}
\put(-81,23){$b$}
\put(-18,71){$d$}
\put(-39,30){$c$}
\put(-15,13){$v$}
\put(-157,58){$u$}
\put(-189,71){$a$}
\put(-188,23){$b$}
\put(-125,71){$d$}
\put(-125,23){$c$}
\put(-157,85){$v$}
\put(-262,58){$u$}
\put(-294,71){$a$}
\put(-293,23){$b$}
\put(-231,71){$d$}
\put(-231,23){$c$}
\put(-262,10){$v$}
\put(-367,58){$u$}
\put(-399,71){$a$}
\put(-398,23){$b$}
\put(-336,71){$d$}
\put(-336,23){$c$}
}
\caption{Four types of short subcycles that can be created by pegging, shown in red.}
\label{fig:3types}
\end{figure}

\noindent\textit{Case 1: Resolving a 3-cycle of the form $ubcu$.}\\
We begin by finding a convenient general form for $C$. Consider the three segments $X$ (between $u$ and $b$), $Y$ (between $b$ and $c$) and $Z$ (between $c$ and $u$). The circuit can then be written as one of $ubcuXbYcZu$ and $ucbuXbYcZu$.  In all figures for this case, green edges and half-edges are on $X$, red on $Y$ and blue on $Z$. We break into three subcases depending on whether one, two or all of the segments have at least $3$ vertices. It is not possible for all three segments to be shorter than this; the smallest 3-connected quartic planar graph excluding the octahedron has 8 vertices, so the average number of vertices on each segment is at least $\frac{10}{3}>3$.

\wl
\noindent\textit{Case 1a.} If $X$, $Y$ and $Z$ all have at least 3 vertices, we can try any of the symmetric circuits shown in Figure~\ref{fig:pretzels}. There are actually six circuits of this form, however we exclude those in which $X$ and $Z$ appear consecutively since these segments are separated by the triangle $ucbu$ in the original circuit.

\begin{figure}[h]
	\small
	\centering
\includegraphics[scale=0.55]{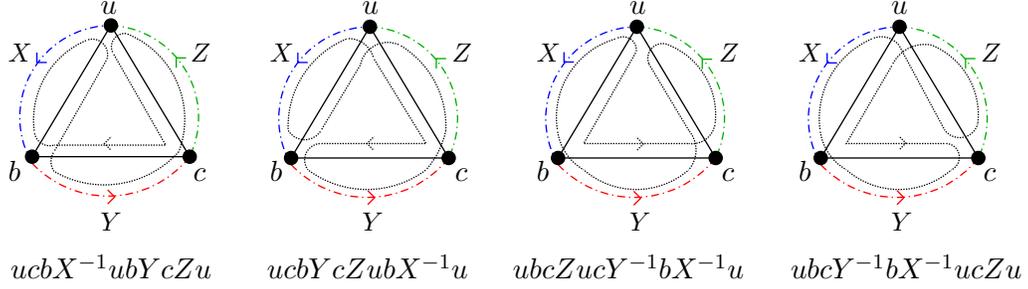}
\put(-370,-27){$ucbX^{-1}ubYcZu$}
\put(-273,-27){$ucbYcZubX^{-1}u$}
\put(-180,-27){$ubcZucY^{-1}bX^{-1}u$}
\put(-75,-27){$ubcY^{-1}bX^{-1}ucZu$}
\put(-37, 73){$u$}
\put(-72, 10){$b$}
\put(-2, 10){$c$}
\put(-72, 55){\footnotesize{$X$}}
\put(-37, -9){\footnotesize{$Y$}}
\put(-3, 55){\footnotesize{$Z$}}
\put(-136, 73){$u$}
\put(-171, 10){$b$}
\put(-101, 10){$c$}
\put(-171, 55){\footnotesize{$X$}}
\put(-136, -9){\footnotesize{$Y$}}
\put(-102, 55){\footnotesize{$Z$}}
\put(-237, 73){$u$}
\put(-272, 10){$b$}
\put(-202, 10){$c$}
\put(-272, 55){\footnotesize{$X$}}
\put(-237, -9){\footnotesize{$Y$}}
\put(-203, 55){\footnotesize{$Z$}}
\put(-336, 73){$u$}
\put(-371, 10){$b$}
\put(-301, 10){$c$}
\put(-371, 55){\footnotesize{$X$}}
\put(-336, -9){\footnotesize{$Y$}}
\put(-302, 55){\footnotesize{$Z$}}
\caption{Possible Eulerian circuits when each segment is long.}
\label{fig:pretzels}
\end{figure}

When we switch between segments, short subcycles may occur in the following situations.
\begin{enumerate}[leftmargin=*]
\item For $ucbX^{-1}ubYcZu$, we have a 3-cycle if $x_1=y_1$ and a 4-cycle if $x_1=y_2$ or $x_2=y_1$. 
\item For $ucbYcZubX^{-1}u$, we have a 3-cycle if $x_{-1}=z_{-1}$ and a 4-cycle if $x_{-1}=z_{-2}$ or $x_{-2}=z_{-1}$. 
\item For $ubcZucY^{-1}bX^{-1}u$, we have a 3-cycle if $y_{-1}=z_{-1}$ and a 4-cycle if $y_{-1}=z_{-2}$ or $y_{-2}=z_{-1}$.
\item For $ubcY^{-1}bX^{-1}ucZu$, we have a 3-cycle if $x_1=z_1$ and a 4-cycle if $x_1=z_2$ or $x_2=z_1$. 
\end{enumerate}

Under our assumptions, 3-cycles can be ruled out. Recall that vertices $a$ and $d$ are also adjacent to $u$, and in particular $\{a,d\} = \{x_1,z_{-1}\}$. If $a=x_1=y_1$, then we would have the edge $y_1b=ab$ in $G$, but this contradicts the definition of pegging. If instead $d=x_1=y_1$, this would contradict the already noted fact that $bd$ is also a non-edge. Similar reasoning rules out the possibility that $x_{-1}=z_{-1}$, $x_{-1}=y_{-1}$ or $x_1=z_1$.

Turning to the 4-cycles, since $G$ is quartic, each vertex not in $\{a,b,c,d,u\}$ appears on exactly two of $X$, $Y$ and $Z$. This means that we cannot, for instance, have both $x_1=y_2$ and $x_1=z_2$. Suppose we try both $ucbX^{-1}ubYcZu$ and $ubcY^{-1}bX^{-1}ucZu$. In order for both of these to have short subcycles, it must be that $x_1=y_2$ and $x_2=z_1$, or $x_1=z_2$ and $x_2=y_1$. We treat these separately. 

\begin{enumerate} 
\item[1ai.] Suppose that $x_1=y_2$ and $x_2=z_1$. If $uby_1x_1u$ is non-facial, then it bounds one region containing $c$ and $x_2$, and another that necessarily contains vertices from two of the three segments in its interior due to parity and 3-connectedness. However, it is impossible for any two distinct segments to have vertices in both regions as $y_1$ is the only vertex at which a segment could cross between them. If $uby_1x_1u$ is facial, then $Y$ cannot be contained inside $bcz_1x_1y_1b$ as then $\{u,z_1\}$ would be a 2-cut, so we have the configuration shown in Figure~\ref{fig:14conflict}(a).This embedding implies that $X$ and $Z$ are disjoint. It follows that $ucbYcZubX^{-1}u$ (Figure~\ref{fig:pretzels} circuit 2) is a good Eulerian circuit as neither $x_{-1}=z_{-2}$ nor $x_{-2}=z_{-1}$ can hold. 
\item[1aii.] Suppose that $x_1=z_2$ and $x_2=y_1$. If $ux_1x_2bu$ is facial, then $z_2z_3$ must occur between $uz_2$ and $z_1z_2$ in the cyclic ordering at $u$, otherwise $\{u,z_1\}$ would be a 2-cut. Then $z_1$ could lie on $Y$ as well as $Z$, or occur in $Z$ twice, forcing the cyclic orderings shown in Figure~\ref{fig:14conflict}(b) and (c) respectively. In both cases, $X$ and $Z$ are disjoint, so $ucbYcZubX^{-1}u$ (Figure~\ref{fig:pretzels} circuit 2) is a good Eulerian circuit. Note that $z_1$ cannot lie on $X$, since then $\{u,x_1\}$ would be a 2-cut. If $ux_1x_2bu$ is non-facial, then to avoid 2-cuts we are forced to have the cyclic ordering shown in Figure~\ref{fig:14conflict}(d). Now $Y$ and $Z$ are disjoint so the circuit $ubcZucY^{-1}bX^{-1}u$ (Figure~\ref{fig:pretzels} circuit 3) will do.
\end{enumerate}

\begin{figure}[h]
	\footnotesize
	\centering
\includegraphics[scale=0.4]{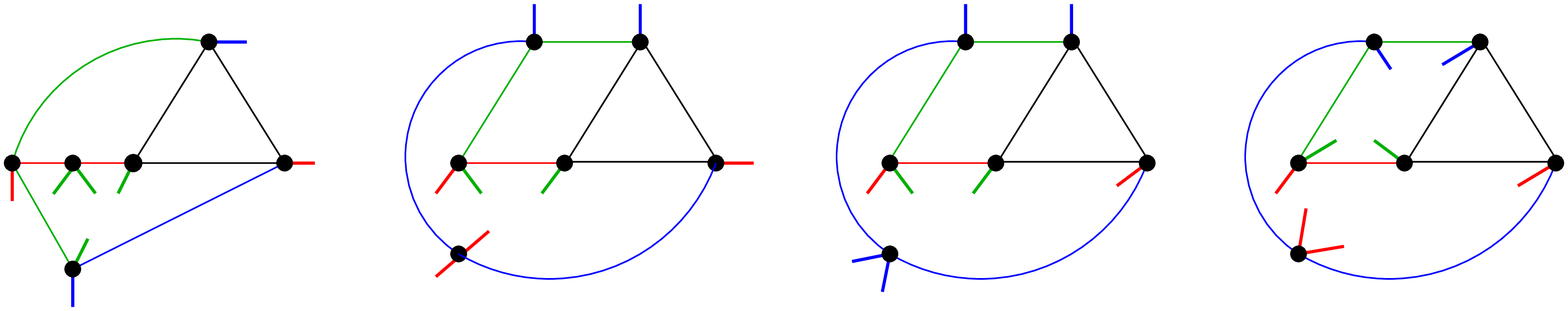}
\put(-22,54){$u$}
\put(-32,37){$b$}
\put(-11,37){$c$}
\put(-53,69){$x_1$}
\put(-70,5){$z_1$}
\put(-62,29){$x_2$}
\put(-42,-8){(d)}
\put(-120,54){$u$}
\put(-130,37){$b$}
\put(-109,37){$c$}
\put(-156,69){$x_1$}
\put(-173,4){$z_1$}
\put(-155,38){$x_2$}
\put(-140,-8){(c)}
\put(-223,54){$u$}
\put(-233,37){$b$}
\put(-212,37){$c$}
\put(-258,69){$x_1$}
\put(-266,4){$z_1$}
\put(-257,38){$x_2$}
\put(-243,-8){(b)}
\put(-326,54){$u$}
\put(-336,37){$b$}
\put(-315,37){$c$}
\put(-362,40){$y_1$}
\put(-383,35){$x_1$}
\put(-369,7){$x_2$}
\put(-341,-8){(a)}
\caption{Cyclic orderings for Cases 1ai and 1aii.}
\label{fig:14conflict}
\end{figure}

\noindent\textit{Case 1b.} Suppose two of $X$, $Y$ and $Z$ have at least 3 vertices, and the third has at most 2 vertices. Let us first assume that $X$ and $Z$ are the longer segments. We again consider the circuit $ucbX^{-1}ubYcZu$ (Figure~\ref{fig:pretzels} circuit 1), which we recall that $x_1\neq y_1$, and this is a good Eulerian circuit unless $x_1=y_2$ or $x_2=y_1$. Combining this with the possible number of vertices in $Y$ leads to three subcases.
\begin{enumerate}
\item[1bi.] If $x_1=y_2$, then $|Y|=2$. Write $X=x_1X'$. Suppose $uz_{-1}$ is in the (outer) region bounded by $ux_1cbu$ that does not contain $y_1$. Then all of $X$ must lie in the interior of this cycle making $\{u,c\}$ a 2-cut, so this cannot be. Suppose instead that $uz_{-1}$ is in the same region bounded by $ux_1cbu$ as $y_1$. Since $z_1$ is adjacent to $c$, it must be that $y_1$ is on both $Z$ and $Y$. In particular, it is not on $X$, which must therefore be contained in either $ux_1y_1bu$ or $by_1x_1cb$. Then either $\{u,y_1\}$ or $\{c,y_1\}$ is a 2-cut. Examples of these configurations are shown in Figure~\ref{fig:2longconflict}(a).

\item[1bii.] If $x_2=y_1$ and $|Y|=1$, then $bx_2c$ must be facial in order to avoid 2-cuts. This determines the cyclic ordering of edges at $c$. Suppose $uz_{-1}$ is between $ux_1$ and $uc$ in the ordering at $u$. If $x_1$ lies on $Z$, then $\{x_1,b\}$ is a 2-cut. Similarly, if the vertex $x_1$ appears twice on $X$, then $\{x_2,b\}$ is a 2-cut. The other possibility is that $uz_{-1}$ occurs between $ux_1$ and $ub$ in the ordering at $u$.  Then $x_1$ must lie on $Z$, and $\{x_1,c\}$ is a 2-cut. These two possible cyclic orderings are shown in Figure~\ref{fig:2longconflict}(b).

\item[1biii.] If $x_2=y_1$ and $|Y|=2$, first observe that we cannot simulateously have $uz_{-1}$ between $ux_1$ and $ub$ in the cyclic ordering at $u$, and $cz_1$ between $bc$ and $y_2c$ in the cyclic ordering at $c$. Let's assume that $cz_1$ is between $uc$ and $y_2c$, and leave the other to symmetry. Writing $X=x_1x_2X'$, if $X'$ has vertices in both regions bounded by the cycle $ux_1x_2bu$ then $\{x_1,b\}$ would be a 2-cut. Thus, $X'$ is in the interior of $ux_1x_2bu$. Figure~\ref{fig:2longconflict}(c) depicts the two possible configurations, and, for a change, neither of these lead to a contradiction. For the upper configuration, we may take the circuit $ubx_2y_2cux_1x_2X'bcZ'x_1Z''u$ where we have split $Z=Z'x_1Z''$. Lemma~\ref{lemma:pathinsquare} ensures that $Z'$ is long enough so that $X'bcZ'x_1Z''$ has no short subcycle. For the lower configuration, the circuit $ucy_2x_2x_1uZ^{-1}cbx_2X'bu$ is 4-locally self-avoiding, again using Lemma~\ref{lemma:pathinsquare} to see that $bx_2X'b$ is not a short subcycle.
\end{enumerate}

\begin{figure}[h]
	\footnotesize
	\centering
\includegraphics[scale=0.49]{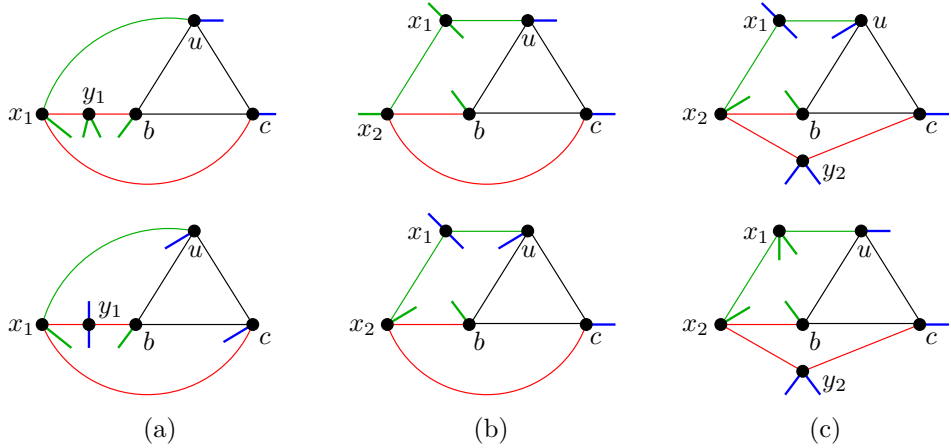}
\put(-30, 140){$u$}
\put(-54, 98){$b$}
\put(-10, 100){$c$}
\put(-79, 140){$x_1$}
\put(-102, 105){$x_2$}
\put(-49, 84){$y_2$}
\put(-36, 53){$u$}
\put(-54, 18){$b$}
\put(-10, 20){$c$}
\put(-79, 60){$x_1$}
\put(-102, 25){$x_2$}
\put(-49, 4){$y_2$}
\put(-163, 133){$u$}
\put(-181, 98){$b$}
\put(-137, 100){$c$}
\put(-206, 140){$x_1$}
\put(-225, 99){$x_2$}
\put(-163, 53){$u$}
\put(-181, 18){$b$}
\put(-137, 20){$c$}
\put(-206, 60){$x_1$}
\put(-228, 25){$x_2$}
\put(-289, 132){$u$}
\put(-306, 98){$b$}
\put(-262, 100){$c$}
\put(-357, 105){$x_1$}
\put(-329, 113){$y_1$}
\put(-289, 53){$u$}
\put(-306, 18){$b$}
\put(-262, 20){$c$}
\put(-357, 25){$x_1$}
\put(-323, 32){$y_1$}
\put(-54, -15){(c)}
\put(-181, -15){(b)}
\put(-306, -15){(a)}
\caption{Configurations in Case 1b.}
\label{fig:2longconflict}
\end{figure}

If $X$ and $Y$ are the longer segments, we can use $ucbYcZubX^{-1}u$ (Figure~\ref{fig:pretzels} circuit 2), and for $Y$ and $Z$ we use $ubcY^{-1}bX^{-1}ucZu$ (Figure~\ref{fig:pretzels} circuit 3), both of which we have already seen in 1a. These are known to avoid 3-cycles already, and an appropriate rotation of the above argument ensures that they also avoid 4-cycles.

\wl
\noindent\textit{Case 1c.} Suppose only one of $X$, $Y$ and $X$ have at least 3 vertices. Since $ab$ and $cd$ are not edges in $G$ and $bc$ is in the cycle, we know that $X$ and $Z$ have at least 2 vertices and $Y$ has at least 1 vertex. Let $X$ and $Y$ be the short subcycles. The same arguments also apply in the other two cases.

Suppose that $X= x_1x_2$ and $Y=y$. Note that if $X$ and $Z$ were our short subcycles, this subcase would not be necessary. The vertices $x_1$, $x_2$ and $y$ are necessarily distinct; certainly $y\neq x_2$ since they are both adjacent to $b$, and setting $y=x_1$ would trap $x_2$. Thus, we have the general configuration shown in Figure~\ref{fig:case1c21}(a). Observe that $ubcu$ and $cbyc$ are necessarily facial, else there would be a 2-cut. Suppose that $ux_1x_2bu$ is non-facial, leading to the cyclic orderings of Figure~\ref{fig:case1c21}(b). Divide $Z$ into two segments with $Z'$ from $c$ to either $x_1$ or $x_2$, and $Z''$ from either $x_1$ or $x_2$ to $u$. The point is to split at the unique vertex out of $x_1$ and $x_2$ that is incident to one edge outside of $ux_1x_2bu$ and one inside. Assume we split at $x_2$ (there are analogous circuits if we split at $x_1$). Then $ubycZ'x_2x_1ucbx_2Z''u$ is a good Eulerian circuit unless $z'_2=y$ or $z'_3=y$. If $Z'_2=y$, then we have the configuration resulting from a 3-cycle slide, so the corresponding reduction can be performed and we are done. Otherwise, let $Z'=z_1z_2Z'''$. Then $ubcyZ'''x_2x_1ucz_1z_2ybx_2Z''u$ is a good Eulerian circuit.

\begin{figure}[h]
	\footnotesize
	\centering
\includegraphics[scale=0.45]{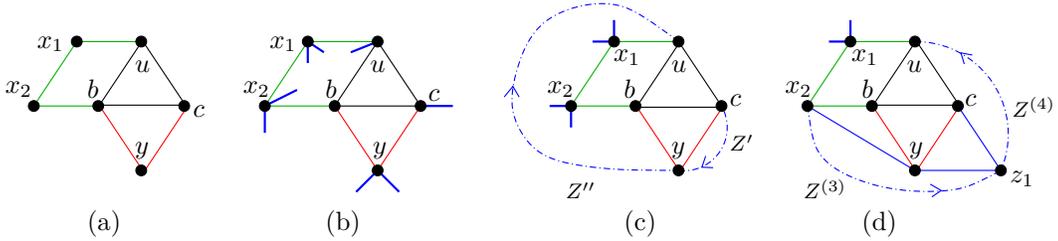}
\put(-39, 48){$u$}
\put(-56, 38){$b$}
\put(-17, 35){$c$}
\put(-60, 52){$x_1$}
\put(-85, 40){$x_2$}
\put(0, 6){$z_1$}
\put(-38, 17){$y$}
\put(1, 31){\scriptsize{$Z^{(4)}$}}
\put(-78,0){\scriptsize{$Z^{(3)}$}}
\put(-56, -12){(d)}
\put(-128, 48){$u$}
\put(-146, 38){$b$}
\put(-106, 35){$c$}
\put(-150, 52){$x_1$}
\put(-175, 40){$x_2$}
\put(-128, 17){$y$}
\put(-106, 18){\scriptsize{$Z'$}}
\put(-168,0){\scriptsize{$Z''$}}
\put(-146, -12){(c)}
\put(-242, 48){$u$}
\put(-259, 38){$b$}
\put(-220, 37){$c$}
\put(-280, 57){$x_1$}
\put(-290, 40){$x_2$}
\put(-241, 17){$y$}
\put(-259, -12){(b)}
\put(-331, 48){$u$}
\put(-349, 38){$b$}
\put(-309, 31){$c$}
\put(-368, 57){$x_1$}
\put(-380, 40){$x_2$}
\put(-331, 17){$y$}
\put(-349, -12){(a)}
\caption{Configurations in Case 1c when $X= x_1x_2$ and $Y=y$.}
\label{fig:case1c21}
\end{figure}

If $ux_1x_2bu$ is facial, split $Z$ at $y$ and write $Z=Z'yZ''$ as depicted in Figure~\ref{fig:case1c21}(c). Consider the circuit $ux_1x_2bcy{Z''}uby{Z'}^{-1}cu$. There are a few potential short subcycles to check. Firstly, if $z''_1=x_2$, there would be a short subcycle in $x_2bcy{Z''}$. Since $ux_2$ is not an edge though (that would trap $x_1$, given the assumption that $ux_1x_2bu$ is facial), either we can apply a 3-cycle unslide centred at $b$ in which case we are done, or $Z'$ has only one vertex, say $z_1$. If the latter is true (Figure~\ref{fig:case1c21}(d)), we can split again at $z_1$ to get $Z''=Z^{(3)}z_1Z^{(4)}$. Then $ux_1x_2bcz_1{Z^{(3)}}^{-1}x_2ybu{Z^{(4)}}^{-1}z_1ycu$ is a good Eulerian circuit unless $Z^{(4)}$ is a single edge, in which case $ux_1x_2yz_1ucbx_2{Z^{(3)}}^{-1}z_1cybu$ suffices. The other potential conflicts occur if $Z''$ is an edge, $z'_{1} = x_1$, $z'_{1} = x_2$, $Z''$ has exactly one vertex, $z'_{1} = x_1$ or $z''_{-1}=z'_{-1}$. The first of these traps $x_1$ and $x_2$, whilst the others create $2$-cuts so we reach contradictions in all cases.

The remaining case is that $X= x_1x_2$ and $Y=y_1y_2$. The four vertices $x_1$, $x_2$, $y_1$ and $y_2$ must all be distinct to avoid trapping other vertices; the possibilities $x_1=y_1$, $x_2=y_2$ and $x_1=y_2$ are shown in Figure~\ref{fig:case122distinct}. 

\begin{figure}[h]
	\footnotesize
	\centering
\includegraphics[scale=0.45]{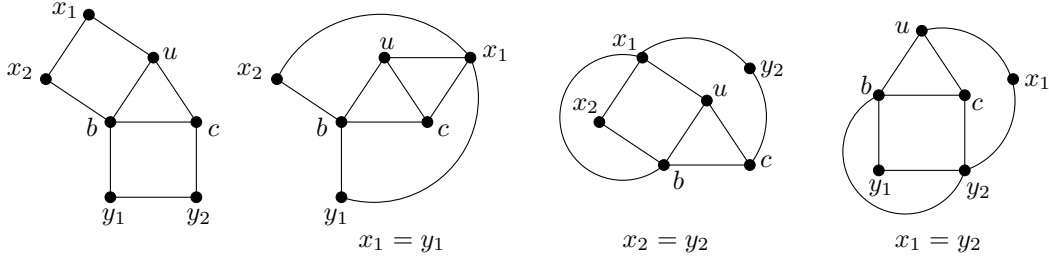}
\put(-324,59){$u$}
\put(-353,30){$b$}
\put(-307,30){$c$}
\put(-347,-2){$y_1$}
\put(-315,-2){$y_2$}
\put(-366,76){$x_1$}
\put(-383,53){$x_2$}
\put(-242,63){$u$}
\put(-266,30){$b$}
\put(-220,30){$c$}
\put(-263,-2){$y_1$}
\put(-203,60){$x_1$}
\put(-296,53){$x_2$}
\put(-250,-12){$x_1=y_1$}
\put(-116,45){$u$}
\put(-131,10){$b$}
\put(-98,18){$c$}
\put(-98,54){$y_2$}
\put(-154,64){$x_1$}
\put(-169,40){$x_2$}
\put(-150,-12){$x_2=y_2$}
\put(-47,69){$u$}
\put(-60,46){$b$}
\put(-18,40){$c$}
\put(-57,9){$y_1$}
\put(-20,8){$y_2$}
\put(2,49){$x_1$}
\put(-47,-12){$x_1=y_2$}
\caption{Traps created when $x_1$, $x_2$, $y_1$ and $y_2$ are not all distinct.}
\label{fig:case122distinct}
\end{figure}

As before, $ubcu$ is necessarily facial. Suppose that both $ux_1x_2bu$ and $by_1y_2cb$ are non-facial. Then $Z$ splits as $Z=Z'y_iZ''x_jZ'''$ where $i,j=1$ or $2$, and $Z'$ and $Z'''$ are contained respectively in the faces $ux_1x_2bu$ and $by_1y_2cb$ (Figure~\ref{fig:case122}(a)). However, this means that $\{x_j,y_i\}$ is a 2-cut, so this configuration cannot occur. 

If $ux_1x_2bu$ is facial but $by_1y_2cb$ is not, we may write $Z=Z'y_iZ''$ with $i=1$ or $2$ so that $Z'$ is contained inside $by_1y_2cb$ (Figure~\ref{fig:case122}(b)). Then $u{Z''}^{-1}y_1y_2cux_1x_2by_1Z'cbu$ is a good Eulerian circuit unless $Z''$ is an edge, but this cannot happen since $x_1$ and $x_2$ would be trapped. We have also used that $Z'$ has at least 5 vertices, which follows from Lemma~\ref{lemma:pathinsquare}. A similar circuit works if $by_1y_2cb$ is facial but $ux_1x_2bu$ is not.

When $ux_1x_2bu$ and $by_1y_2cb$ are both facial (Figure~\ref{fig:case122}(c)), reroute to $ux_1x_2bcZuby_1y_2cu$. Note that $Z$ must contain $x_1$, $x_2$, $y_1$ and $y_2$ so $bcZub$ will not be a short subcycle. The only way a short subcycle can be present here is if $z_{1}=x_2$, $z_{2}=x_2$, $z_{-1}=y_1$ or $z_{-2}=y_1$. The first and third of these trap either $x_1$ or $y_2$. For the second, the situation is as shown in Figure~\ref{fig:case122}(d). Here, a good circuit is given by $uby_1y_2cz_1Z''ucbx_2Z'z_1x_2x_1u$, where we have split $Z$ as $z_1x_2Z'z_1Z''$. The fourth case is symmetric to the second, and a good circuit is given by $ucy_2y_1z_{-1}ux_1x_2bcZ'z_{-1}Z''bu$ where $Z=  Z'z_{-1}Z''y_1z_{-1}$ to separate the segments contained in different faces.

\vspace{3mm}
\begin{figure}[h]
	\footnotesize
	\centering
\includegraphics[scale=0.45]{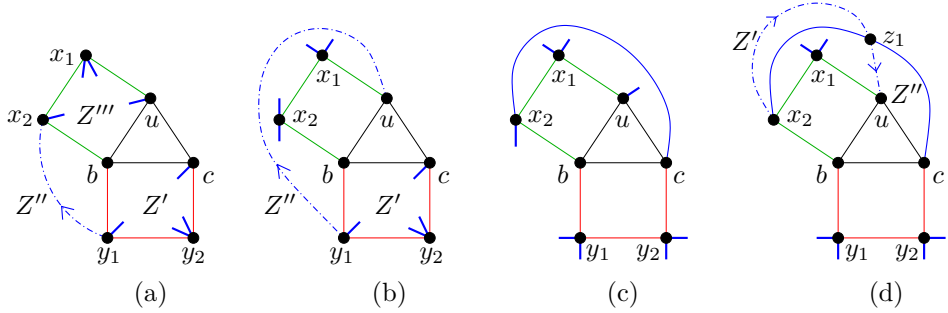}
\put(-28,51){$u$}
\put(-50,30){$b$}
\put(-6,30){$c$}
\put(-39,2){$y_1$}
\put(-20,2){$y_2$}
\put(-25,83){$z_1$}
\put(-52,68){$x_1$}
\put(-61,53){$x_2$}
\put(-82,80){\footnotesize{$Z'$}}
\put(-22,60){\footnotesize{$Z''$}}
\put(-30,-15){(d)}
\put(-126,51){$u$}
\put(-148,30){$b$}
\put(-104,30){$c$}
\put(-137,2){$y_1$}
\put(-118,2){$y_2$}
\put(-150,68){$x_1$}
\put(-159,53){$x_2$}
\put(-129,-15){(c)}
\put(-215,51){$u$}
\put(-237,30){$b$}
\put(-193,30){$c$}
\put(-234,1){$y_1$}
\put(-200,1){$y_2$}
\put(-239,68){$x_1$}
\put(-248,53){$x_2$}
\put(-217,18){\footnotesize{$Z'$}}
\put(-259,18){\footnotesize{$Z''$}}
\put(-218,-15){(b)}
\put(-304,51){$u$}
\put(-326,30){$b$}
\put(-282,30){$c$}
\put(-322,1){$y_1$}
\put(-290,1){$y_2$}
\put(-340,76){$x_1$}
\put(-356,53){$x_2$}
\put(-305,18){\footnotesize{$Z'$}}
\put(-353,18){\footnotesize{$Z''$}}
\put(-330,52){\footnotesize{$Z'''$}}
\put(-308,-15){(a)}
\caption{Configurations in Case 1c when $X=x_1x_2$ and $Y=y_1y_2$.}
\label{fig:case122}
\end{figure}

\noindent\textit{Case 2: Resolving a 4-cycle of the form $ubvcu$}\\
If we ignore the edges in the cycle $ubvcu$, the Eulerian circuit in the rest of the graph is of form $uXbcYvZu$, $uXvYbcZu$ or $uXvYcbZu$, assuming that we visit either $b$ or $v$ before $d$ since orientation is unimportant. These are shown in Figure~\ref{fig:case2}. Again in figures for this case, we will use green to represent $X$, red for $Y$, and blue for $Z$. The second and third circuits can be obtained from the first by reversing the circuit and possibly reflecting the whole picture, so it is enough to assume the form $uXbcYvZu$. Since $ab$ and $bd$ are non-edges, we know that $X$ has at least 2 vertices. 

\begin{figure}[h]
	\small
	\centering
\includegraphics[scale=0.65]{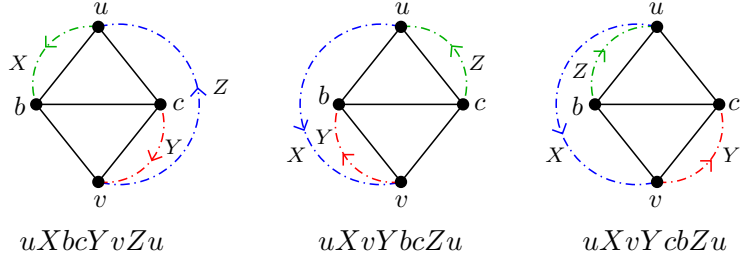}
\put(-268,-23){$uXbcYvZu$}
\put(-155,-23){$uXvYbcZu$}
\put(-55,-23){$uXvYcbZu$}
\put(-30,66){\footnotesize{$u$}}
\put(-30,-7){\footnotesize{$v$}}
\put(-59,28){\footnotesize{$b$}}
\put(0,30){\footnotesize{$c$}}
\put(-69,10){\scriptsize{$X$}}
\put(-2,10){\scriptsize{$Y$}}
\put(-59,42){\scriptsize{$Z$}}
\put(-126,66){\footnotesize{$u$}}
\put(-126,-7){\footnotesize{$v$}}
\put(-155,32){\footnotesize{$b$}}
\put(-96,30){\footnotesize{$c$}}
\put(-167,10){\scriptsize{$X$}}
\put(-156,17){\scriptsize{$Y$}}
\put(-98,45){\scriptsize{$Z$}}
\put(-240,66){\footnotesize{$u$}}
\put(-240,-7){\footnotesize{$v$}}
\put(-270,28){\footnotesize{$b$}}
\put(-210,30){\footnotesize{$c$}}
\put(-195,36){\scriptsize{$Z$}}
\put(-213,13){\scriptsize{$Y$}}
\put(-272,45){\scriptsize{$X$}}
\caption{Circuit configurations with a short subcycle of the form $ubvcu$.}
\label{fig:case2}
\end{figure}

\noindent\textit{Case 2a.} Suppose that $X$ has at least 3 vertices. Then we claim that $ubcYvZucvbX^{-1}u$ is a good Eulerian circuit. The only way that this may fail is if $Z$ has fewer than two vertices. To rule this out, we note that if $Z$ has one vertex as illustrated in Figure~\ref{fig:2ab}(a), then either $z_1$ lies on $X$ making $\{z_1,b\}$ or $\{c,v\}$ a 2-cut, or else it lies on $Y$ but then $\{c,z_1\}$ or $\{u,b\}$ is a 2-cut. If $Z =vu$ were an edge as in Figure~\ref{fig:2ab}(b), then $X$ and $Y$ would be on opposite regions bounded by $ubvu$ and so both $\{c,v\}$ and $\{u,b\}$ would be 2-cuts.

\wl
\noindent\textit{Case 2b.} We now assume that $X=x_1x_2$ and break into three subcases depending on the length of $Y$.
\begin{enumerate}
\itemsep=0mm
\item[2bi.] If $Y$ has at least 3 vertices, consider the circuit $ubcux_1x_2bvcYvZu$. Since having $y_1=x_2$ would obstruct $Z$ (Figure~\ref{fig:2ab}(c)), we can be sure that $x_2bvcy_1$ is not a subcycle. Apart from $ubcu$, which occurs in the new circuit by design, there are no other potential local conflicts. Thus, we are done by case 1.

\item[2bii.] If $Y=y_1y_2$, then we first verify that these are distinct from $x_1$ and $x_2$. Indeed, both $x_1=y_1$ and $x_2=y_2$ obstruct $Z$, whilst $x_1=y_2$ and $x_1=y_1$ create 2-cuts. Hence, the configuration is that shown in Figure~\ref{fig:2ab}(d). We propose the circuit $ux_1x_2bcvZubvy_2y_1cu$. If $z_{-1}=y_2$ then we would have a subcycle $z_{-1}ubvy_2$, but this traps either $y_1$, or $x_1$ and $x_2$. The only other potential short subcycle is $x_2bcvz_1$. If $z_1=x_2$, then a 3-cycle unslide centred at $b$ can be applied to the bold edges of Figure~\ref{fig:2ab}(e); this requires firstly that $ux_2 \not\in E(G)$, and secondly that $c$ and $v$ have only $b$ as a common neighbour. The former holds since the circuit is Eulerian, and for the latter, given that $N(c) = \{u,b,v,y_1\}$ and $N(v) = \{b,c,y_2, z\}$, distinctness can be easily checked.

\begin{figure}[h]
	\footnotesize
	\centering
\scalebox{0.93}{
\begin{subfigure}[t]{\textwidth}
	\centering
\includegraphics[scale=0.5]{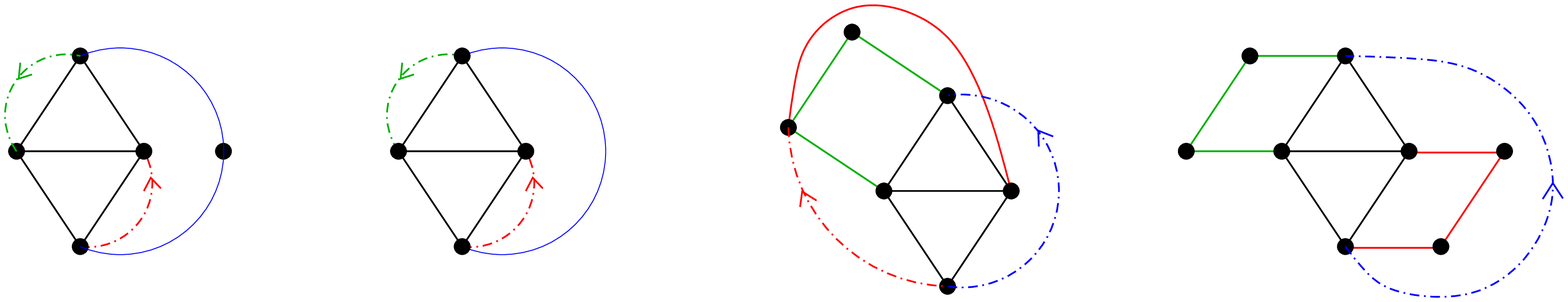}
\put(-178, 62){$u$}
\put(-202, 24){$b$}
\put(-154, 24){$c$}
\put(-178, -7){$v$}
\put(-205, 65){$x_1$}
\put(-216, 47){$x_2$}
\put(-216, 9){\scriptsize{$Y'$}}
\put(-147, 10){\scriptsize{$Z$}}
\put(-316, 74){$u$}
\put(-341, 37){$b$}
\put(-291, 39){$c$}
\put(-318, 4){$v$}
\put(-293, 20){\scriptsize{$Y$}}
\put(-342, 63){\scriptsize{$X$}}
\put(-424, 74){$u$}
\put(-449, 37){$b$}
\put(-399, 39){$c$}
\put(-426, 4){$v$}
\put(-401, 20){\scriptsize{$Y$}}
\put(-380, 34){$z_1$}
\put(-450, 63){\scriptsize{$X$}}
\put(-65, 74){$u$}
\put(-72, 45){$b$}
\put(-44, 46){$c$}
\put(-68, 5){$v$}
\put(-22, 48){$y_1$}
\put(-34,6){$y_2$}
\put(-104, 68){$x_1$}
\put(-122, 40){$x_2$}
\put(-7,11){\scriptsize{$Z$}}
\put(-414, -20){(a)}
\put(-315, -20){(b)}
\put(-195, -20){(c)}
\put(-70, -20){(d)}
\vspace{5mm}
\end{subfigure}}
\scalebox{0.93}{
\begin{subfigure}[t]{\textwidth}
	\centering
\includegraphics[scale=0.5]{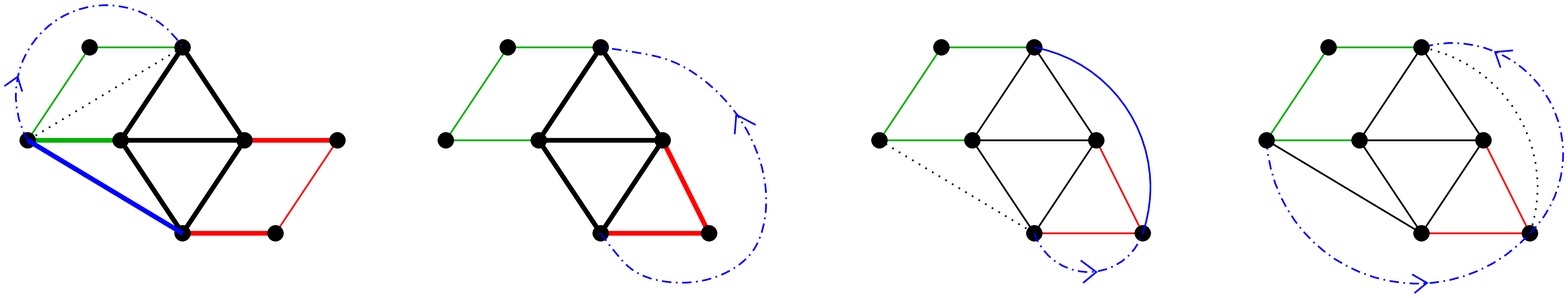}
\put(-45, 77){$u$}
\put(-64, 50){$b$}
\put(-25, 49){$c$}
\put(-45, 8){$v$}
\put(-12,9){$y$}
\put(-83, 71){$x_1$}
\put(-102, 43){$x_2$}
\put(-88,15){\scriptsize{$Z'$}}
\put(-12,70){\scriptsize{$Z''$}}
\put(-45, -12){(h)}
\put(-158, 77){$u$}
\put(-177, 49){$b$}
\put(-138, 50){$c$}
\put(-158, 8){$v$}
\put(-125,9){$y$}
\put(-196, 71){$x_1$}
\put(-207, 35){$x_2$}
\put(-139,-2){\scriptsize{$Z'$}}
\put(-177, -12){(g)}
\put(-283, 77){$u$}
\put(-303, 49){$b$}
\put(-264, 50){$c$}
\put(-284, 8){$v$}
\put(-251,9){$y$}
\put(-322, 71){$x_1$}
\put(-333, 35){$x_2$}
\put(-232,18){\scriptsize{$Z$}}
\put(-303, -12){(f)}
\put(-399, 73){$u$}
\put(-414, 48){$b$}
\put(-385, 50){$c$}
\put(-406, 8){$v$}
\put(-363, 51){$y_1$}
\put(-380,9){$y_2$}
\put(-444, 72){$x_1$}
\put(-456, 36){$x_2$}
\put(-423, -12){(e)}
\end{subfigure}}
\caption{Configurations in Cases 2a and 2bi. The thicker edges identify subgraphs where a reduction operation might be applied, and dotted lines represent non-edges.}
\label{fig:2ab}
\end{figure}

\vspace{-2mm}
\item[2biii.] If $Y$ has a single vertex $y$, then the three adjacent triangles shown in bold in Figure~\ref{fig:2ab}(f) suggests that we could try to apply a 3-cycle unslide centred at $c$. This can be done provided $z_{-1}\neq y$ and $z_1 \neq x_2$. If both $z_{-1}=y$ and $z_1=x_2$, then $ucyvbx_2Z'ubcvx_2x_1u$ is a good Eulerian circuit, where $Z=x_2Z'$ and $Z'$ has at least one vertex $x_1$. If $z_{-1}=y$ and $z_1\neq x_2$ (Figure~\ref{fig:2ab}(g)), then with $Z=Z'y$, a good circuit is given by $ux_1x_2bcvZ'ycubvyu$. Note that $cvZ'yc$ is not a short subcycle since $Z'$ contains at least $x_1$ and $x_2$. Finally, if $z_{-1}\neq y$ and $z_1= x_2$, let's split $Z$ into $x_2Z'yZ''$ (Figure~\ref{fig:2ab}(h)) and try $ux_1x_2vcyZ''ubvy{Z'}^{-1}x_2bcu$. Since $uy$ is not in $G$, it follows that $Z''$ has at least one vertex so $yZ''ubvy$ is not a problem. However, $Z'$ may be an edge. If it is, we can avoid the subcycle $bvyx_2b$ by instead rerouting to $ux_1x_2ycvbu{Z''}^{-1}yvx_2bcu$ which now has no possible conflicts.
\end{enumerate}

\noindent\textit{Case 3: Resolving a 4-cycle of the form $uavdu$}\\
We may assume that $ad$ is a non-edge, otherwise it is symmetric to Case 2. Note that for this short subcycle to occur, the induced circuit must contain the subpath $buavduc$. It is actually convenient to forget about the edge $bc$ being present here.

Suppose firstly that $uavdu$ is facial. The induced circuit can be written as $buavducXvYb$ with cyclic ordering as in Figure~\ref{fig:upper4cycle}(a). Consider the rerouted circuit $buavX^{-1}cudvYb$. The only possible conflicts occur in $buavX^{-1}$ and $cudvY$. These are symmetric, so we will just handle the first case. We know that $x_{-1} \neq b$ since this would trap $a$. However, it is possible that $x_{-2} = a$ or $x_{-3} = a$ which would create a 3-cycle or 4-cycle respectively. If a 3-cycle is created, say $avx_{-1}a$, then reroute to $buax_{-1}vducX'avYb$ where $X=X'ax_{-1}$,. The only place for short subcycles is in $X'avY$, but the bounded region of a short subcycle here would contain $avx_{-1}a$, which is necessarily facial, and trap $x_{-1}$. That settles the 3-cycle case. 

If $x_{-3} = a$ and we have a 4-cycle $avx_{-1}x_{-2}a$, an analogous circuit $buax_{-2}x_{-1}vducX'avYb$ is good provided $avx_{-1}x_{-2}a$ is facial by the same reasoning as before. If it is not facial, then the same circuit fails only when $x_{-4} = y_1$, $x_{-4} = y_2$, or $x_{-5} = y_1$. These create short subcycles in the region bounded by $avx_{-1}x_{-2}a$, so in particular $X$ and $Y$ both have vertices in this region and $X$ contains either $x_{-1}$ or $x_{-2}$ twice. We provide a good Eulerian circuit for the first case only, and assume additionally that $x_{-2}$ occurs twice in $X$ so we can write $X=X'x_{-2}X''y_1ax_{-2}x_{-1}$ (Figure~\ref{fig:upper4cycle}(b)), namely $buay_1{X''}^{-1}x_{-2}x_{-1}vducX'x_{-2}avy_1Yb$. The others possibilities admit very similar circuits, likewise keeping $ducX'$ and $Ybua$ from the original circuit to avoid any trouble at $b$, $c$, and $d$.

If instead $uavdu$ is non-facial, there are three cases depending on which one of $a$, $d$ and $v$ have both half-edges in the same region bounded by $uavdu$, although the first two are symmetric. In the case when it is $v$, the possible induced circuits are $buavducXdYaZb$, which is shown in Figure~\ref{fig:upper4cycle}(c), and $buavducXaYdZb$. The first can be rerouted to the good circuit $buaY^{-1}dX^{-1}cudvaZb$, and second to $budY^{-1}aX^{-1}cuavdZb$ which is completely analogous. When $a$ (symmetrically $d$) has both half-edges pointing inside $uavdu$, the possible induced circuits are $buavducXdYvZb$, shown in Figure~\ref{fig:upper4cycle}(d), and $buavducXvYdZb$. These can be rerouted to the good circuits $buavdX^{-1}cudYvZb$ and $buavX^{-1}cudY^{-1}vdZb$ respectively. In the second circuit, Lemma~\ref{lemma:pathinsquare} ensures that $dY^{-1}vd$ is not a short subcycle.

\vspace{2mm}
\begin{figure}[h]
	\small
	\centering
\includegraphics[scale=0.5]{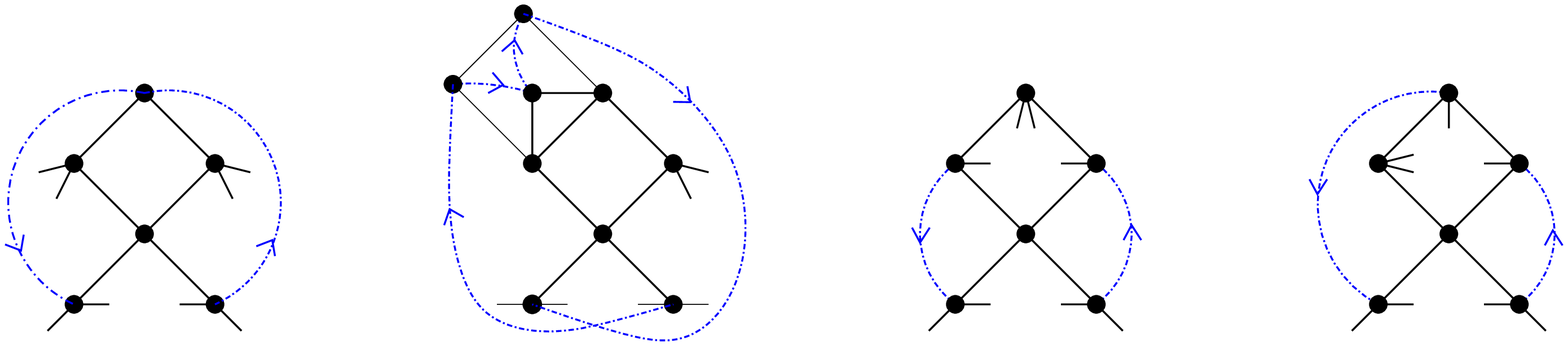}
\put(-57, 44){\footnotesize{$a$}}
\put(-50, 15){\footnotesize{$b$}}
\put(-14, 15){\footnotesize{$c$}}
\put(-8, 44){\footnotesize{$d$}}
\put(-33, 33){\footnotesize{$u$}}
\put(-33, 69){\footnotesize{$v$}}
\put(1, 24){\footnotesize{$X$}}
\put(-73, 24){\footnotesize{$Y$}}
\put(-165, 44){\footnotesize{$a$}}
\put(-158, 15){\footnotesize{$b$}}
\put(-122, 15){\footnotesize{$c$}}
\put(-116, 44){\footnotesize{$d$}}
\put(-141, 33){\footnotesize{$u$}}
\put(-141, 69){\footnotesize{$v$}}
\put(-109, 24){\footnotesize{$X$}}
\put(-178, 24){\footnotesize{$Z$}}
\put(-273, 42){\footnotesize{$a$}}
\put(-267, 15){\footnotesize{$b$}}
\put(-231, 15){\footnotesize{$c$}}
\put(-227, 48){\footnotesize{$d$}}
\put(-250, 33){\footnotesize{$u$}}
\put(-250, 69){\footnotesize{$v$}}
\put(-287, 83){\footnotesize{$x_{-1}$}}
\put(-300, 73){\footnotesize{$x_{-2}$}}
\put(-220, 65){\footnotesize{$Y$}}
\put(-281, 25){\footnotesize{$X$}}
\put(-389, 49){\footnotesize{$a$}}
\put(-384, 15){\footnotesize{$b$}}
\put(-348, 15){\footnotesize{$c$}}
\put(-344, 47){\footnotesize{$d$}}
\put(-367, 33){\footnotesize{$u$}}
\put(-367, 69){\footnotesize{$v$}}
\put(-329, 24){\footnotesize{$X$}}
\put(-409, 24){\footnotesize{$Z$}}
\put(-371,-15){(a)}
\put(-253,-15){(b)}
\put(-146,-15){(c)}
\put(-37,-15){(d)}
\caption{Configurations in Case 3.}
\label{fig:upper4cycle}
\end{figure}

\noindent\textit{Case 4: Resolving a 4-cycle of the form $ubvdu$}\\
This case arises when $C$ has $aubvduc$ as a subpath. Consider the $(c,u)$-trail obtained by removing the edges of $ubvduc$ from $C$, which can be broken into four segments by splitting at each each of $b$, $v$ and $d$. To avoid 2-cuts, there must be two segments in each of the regions bounded by $ubvdu$. We know that $a$ is outside region, so segments inside are between $b$, $c$, $d$, and $v$. To enumerate the cases, we first list the possible pairs of segments that could be inside the short subcycle which are restricted by 4-regularity. These are
\begin{itemize}[noitemsep,topsep=1pt]
\item a $(c,c)$-trail together with a $(b,d)$-trail, $(b,v)$-trail, or $(b,d)$-trail, or
\item any two out of a $(c,d)$-trail, $(c,v)$-trail and $(b,v)$-trail
\end{itemize}
making six possibilities in total. For each of these, we then pair up the degree deficient vertices on $ubvdu$ in all possible ways so that each degree is 4. At this point, we have a collection of 9 diagrams that specify the ends of each segment and the cyclic ordering at each vertex in the short subcycle. These are given in Figure~\ref{fig:case4}. The arrows on the segments indicate one possible induced circuit for each embedding. The circuit is unique up to reversal for (a), (b), (c), and (d), but the last five segment configurations each give rise to two possible forms for the induced circuit.

\begin{figure}[h]
	\small
	\centering
	\scalebox{0.88}{
\includegraphics[scale=0.5]{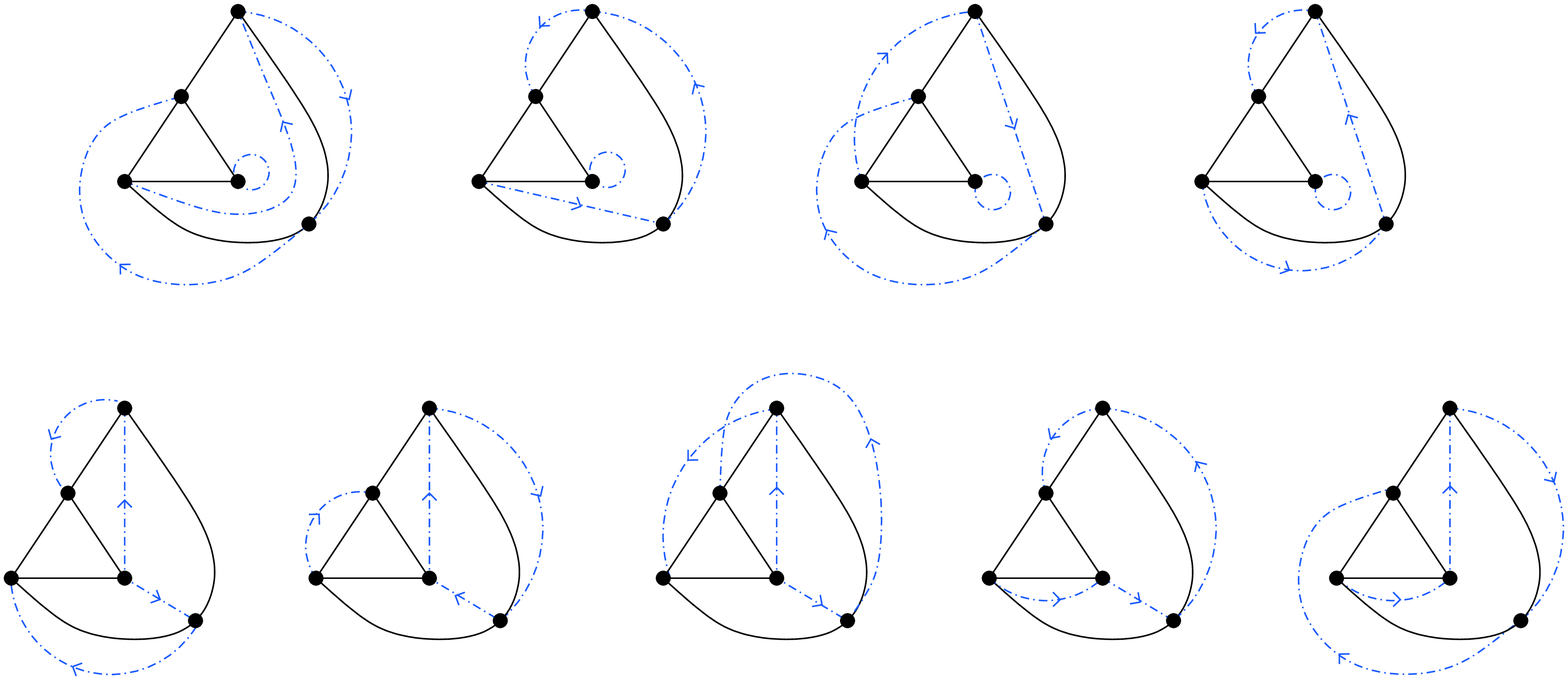}
\put(-444, 176){\footnotesize{$u$}}
\put(-452, 161){\footnotesize{$b$}}
\put(-433, 161){\footnotesize{$c$}}
\put(-426, 218){\footnotesize{$d$}}
\put(-399, 137){\footnotesize{$v$}}
\put(-424, 170){\scriptsize{$W$}}
\put(-424, 187){\scriptsize{$X$}}
\put(-390, 196){\scriptsize{$Y$}}
\put(-462, 137){\scriptsize{$Z$}}
\put(-432,110){(a)}
\put(-331, 176){\footnotesize{$u$}}
\put(-339, 161){\footnotesize{$b$}}
\put(-320, 161){\footnotesize{$c$}}
\put(-312, 218){\footnotesize{$d$}}
\put(-286, 137){\footnotesize{$v$}}
\put(-311, 170){\scriptsize{$W$}}
\put(-315, 142){\scriptsize{$X$}}
\put(-277, 196){\scriptsize{$Y$}}
\put(-341, 197){\scriptsize{$Z$}}
\put(-322,110){(b)}
\put(-209, 176){\footnotesize{$u$}}
\put(-217, 161){\footnotesize{$b$}}
\put(-198, 161){\footnotesize{$c$}}
\put(-191, 218){\footnotesize{$d$}}
\put(-164, 137){\footnotesize{$v$}}
\put(-189, 165){\scriptsize{$W$}}
\put(-228, 196){\scriptsize{$X$}}
\put(-189, 187){\scriptsize{$Y$}}
\put(-227, 137){\scriptsize{$Z$}}
\put(-202,110){(c)}
\put(-101, 176){\footnotesize{$u$}}
\put(-109, 161){\footnotesize{$b$}}
\put(-90, 161){\footnotesize{$c$}}
\put(-82, 218){\footnotesize{$d$}}
\put(-56, 137){\footnotesize{$v$}}
\put(-81, 165){\scriptsize{$W$}}
\put(-119, 137){\scriptsize{$X$}}
\put(-81, 187){\scriptsize{$Y$}}
\put(-110, 196){\scriptsize{$Z$}}
\put(-92,110){(d)}
\put(-480, 49){\footnotesize{$u$}}
\put(-488, 34){\footnotesize{$b$}}
\put(-469, 34){\footnotesize{$c$}}
\put(-460, 91){\footnotesize{$d$}}
\put(-434, 12){\footnotesize{$v$}}
\put(-449, 31){\scriptsize{$W$}}
\put(-492, 0){\scriptsize{$X$}}
\put(-457, 60){\scriptsize{$Y$}}
\put(-492, 70){\scriptsize{$Z$}}
\put(-470,-15){(e)}
\put(-383, 49){\footnotesize{$u$}}
\put(-391, 34){\footnotesize{$b$}}
\put(-372, 34){\footnotesize{$c$}}
\put(-364, 91){\footnotesize{$d$}}
\put(-337, 12){\footnotesize{$v$}}
\put(-409, 52){\scriptsize{$Z$}}
\put(-352, 31){\scriptsize{$Y$}}
\put(-330, 70){\scriptsize{$X$}}
\put(-360, 60){\scriptsize{$W$}}
\put(-370,-15){(f)}
\put(-272, 49){\footnotesize{$u$}}
\put(-280, 34){\footnotesize{$b$}}
\put(-261, 34){\footnotesize{$c$}}
\put(-246, 86){\footnotesize{$d$}}
\put(-226, 12){\footnotesize{$v$}}
\put(-249, 60){\scriptsize{$W$}}
\put(-298, 52){\scriptsize{$X$}}
\put(-241, 31){\scriptsize{$Y$}}
\put(-217, 70){\scriptsize{$Z$}}
\put(-260,-15){(g)}
\put(-168, 49){\footnotesize{$u$}}
\put(-176, 34){\footnotesize{$b$}}
\put(-157, 34){\footnotesize{$c$}}
\put(-149, 91){\footnotesize{$d$}}
\put(-122, 12){\footnotesize{$v$}}
\put(-158, 18){\scriptsize{$W$}}
\put(-137, 30){\scriptsize{$X$}}
\put(-128, 65){\scriptsize{$Y$}}
\put(-175, 70){\scriptsize{$Z$}}
\put(-156,-15){(h)}
\put(-58, 49){\footnotesize{$u$}}
\put(-66, 34){\footnotesize{$b$}}
\put(-47, 34){\footnotesize{$c$}}
\put(-39, 91){\footnotesize{$d$}}
\put(-12, 12){\footnotesize{$v$}}
\put(-48, 18){\scriptsize{$W$}}
\put(-35, 50){\scriptsize{$X$}}
\put(-18, 65){\scriptsize{$Y$}}
\put(-72, 12){\scriptsize{$Z$}}
\put(-46,-15){(i)}
}
\caption{Segment configurations from the induced circuit when $ubvdu$ is a subcycle.}
\label{fig:case4}
\end{figure}

The remainder of the work is in rerouting each possible circuit. We follow the order shown in Figure~\ref{fig:case4}.
\begin{enumerate}[topsep=0pt]
\itemsep=1mm
\item[4a.] We reroute $ubvducWcbXdYvZu$ to the good circuit $udvbXdYvZubcWcu$. Note that $X$ has at least two vertices since otherwise $c$ would either be a cutvertex or $\{c,x_1\}$ would be a 2-cut, so $dvbXd$ is not a short subcycle.

\item[4b.] Starting with $ubvducWcbXvYdZu$, reroute to the good circuit $ucWcbvYdubXvdZu$.

\item[4c.] Similar to 4b, the circuit $ubvducWcbXdYvZu$ can be rerouted to $ubXdYvZudvbcWcu$ which is good.

\item[4d.] Starting with an induced circuit of the form $ubvducWcbXvYdZu$, first reroute to $ubXvYducWcbvdZu$. If $y_{-1}=w_1$, then $y_1ducw_1$ would be a short subcycle. By rerouting again to $ubvducw_1{Y'}^{-1}vX^{-1}bc{W'}^{-1}w_1dZu$ where $W=w_1W'$ and $Y=Y'y_{-1}$, we now have a different circuit with one short subcycle $ubvdu$ which is of the form handled in 4e. By deferring that case, we may assume here that $y_{-1}\neq w_1$. The other potential conflicts occur if there is a short cycle in $ZubX$. This can't be a 3-cycle, since that would require $x_1=z_{-1}=a$ to be adjacent to $d$ which is impossible after a pegging, but there could be a 4-cycle if $x_1=z_{-2}$ or $x_2=z_{-1}$. If either of these hold, we must have at least 3 vertices on $Z$ or $X$ respectively to avoid trapping $z_{-1}$ or $x_1$, so we can instead use the good circuits $ubXvYdZudvbcWcu$ or $ubcWcudvbXvYdZu$.

\item[4e.] From the figure, we have the circuit $ubvducWvXbcYdZu$. The other circuit with segments of the same form can be obtained by reversing the subcircuit $cWvXbc$. Both of these can be rerouted to the good circuit $ucWvXbudY^{-1}cbvdZu$. The key observations to verify this are that $W$ and $Y$ must have at least one vertex to avoid the endpoints of the other segment being a 2-cut, and $uv$ cannot be an edge. These ensure that $budY^{-1}cb$ and $ZucWvX$ do not have short subcycles, for example.

\item[4f.] We have the circuit $ubvducWdXvYcbZu$, and reversing $cWdXvYc$ gives the other possibility. Both of these can be rerouted to the good circuit $uZ^{-1}bvdW^{-1}cbudXvYcu$, and the verification is essentially the same as 4e.

\item[4g.] We have the circuit $ubvducWdXbcYvZu$, and one more given by reversing $cWdXbc$. Both of these can be rerouted to the good circuit $ucYvZudvbcWdXbu$, noting that $W$ and $Y$ have at least one vertex, and $X$ and $Z$ have at least two vertices.

\item[4h.] We have the circuit $ubvducbWcXvYdZu$, which can be rerouted to the good circuit $ubWcbvYdZucXvdu$. The other induced circuit with this segment configuration is given by reversing $cbWc$, and the same solution holds.

\item[4i.] Finally, we have an induced circuit of the form $ubvducbWcXdYvZu$, and again the other circuit is given by reversing $cbWc$. In both of these, each segment has at least one vertex and a good circuit is hence given by $ucYvZudvbcWdXbu$.
\end{enumerate}
\vspace{-1em}
\end{proof}

\begin{figure}[h]
	\footnotesize
	\centering
	\scalebox{0.8}{
\includegraphics[scale=0.5]{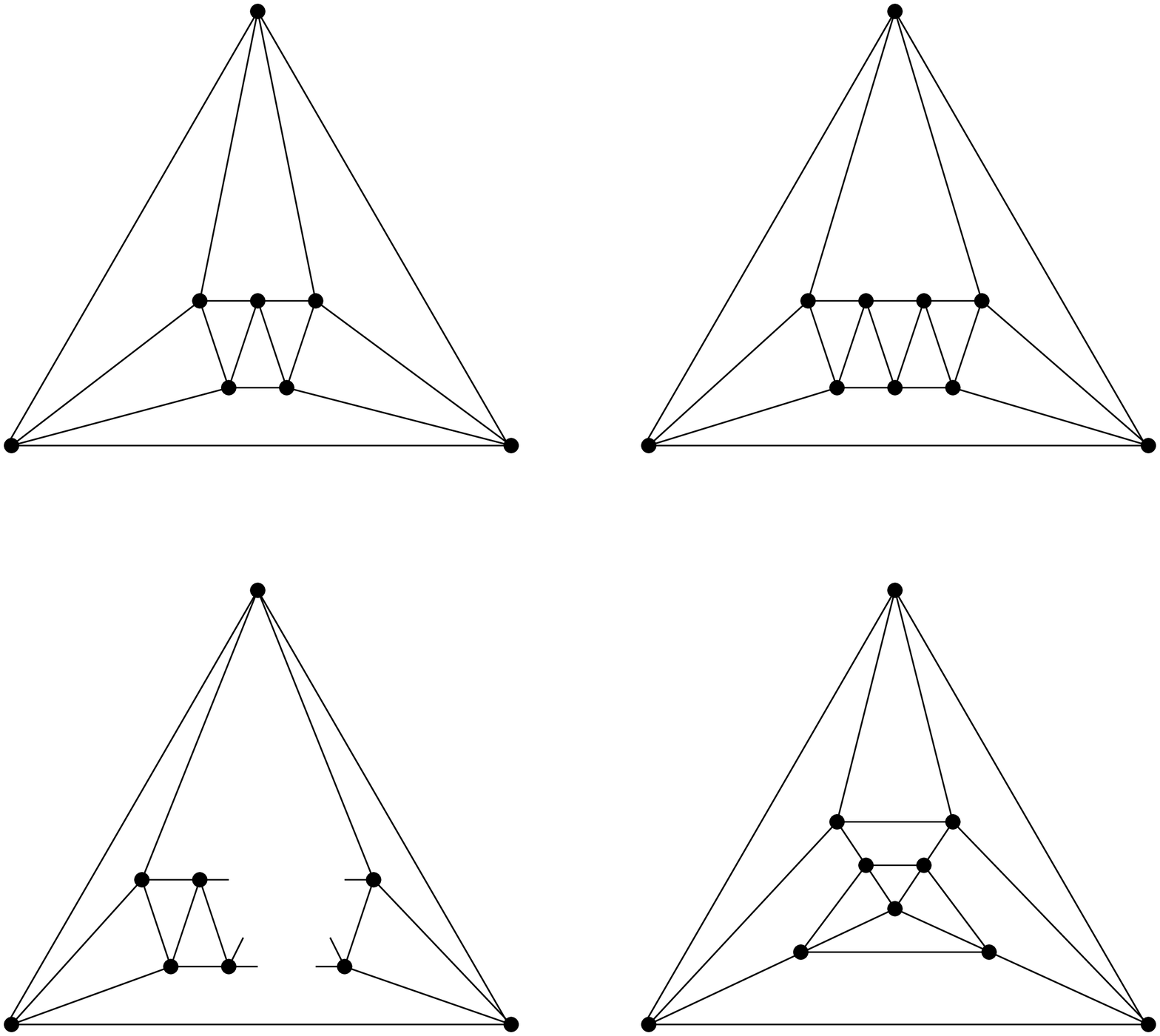}
\put(-283,324){$a$}
\put(-365,172){$b$}
\put(-200,173){$c$}
\put(-313,229){$x_1$}
\put(-285,234){$x_2$}
\put(-258,229){$x_3$}
\put(-295,192){$y_1$}
\put(-275,192){$y_2$}
\put(-84,324){$a$}
\put(-166,172){$b$}
\put(-1,173){$c$}
\put(-123,229){$x_1$}
\put(-95,234){$x_2$}
\put(-76,234){$x_3$}
\put(-51,229){$x_4$}
\put(-105,192){$y_1$}
\put(-86,192){$y_2$}
\put(-67,192){$y_3$}
\put(-283,144){$a$}
\put(-365,-8){$b$}
\put(-200,-7){$c$}
\put(-325,53){$x_1$}
\put(-303,53){$x_2$}
\put(-277,38){$\ldots$}
\put(-270,53){$x_{n-1}$}
\put(-312,11){$y_1$}
\put(-295,11){$y_2$}
\put(-260,11){$y_{n-2}$}
\put(-84,144){$a$}
\put(-166,-8){$b$}
\put(-1,-7){$c$}
\put(-110,65){$d$}
\put(-84,29){$e$}
\put(-59,65){$f$}
\put(-100,51){$v$}
\put(-117,15){$w$}
\put(-53,15){$x$}
\put(-68,51){$y$}
}
\caption{Graphs isomorphic to the 4-antiprism, the 5-antiprism, the $n$-antiprism, and the octahedron with a single 4-cycle addition.}
\label{basecase}
\end{figure}

\begin{proof}[Proof of Theorem~\ref{thm:no34_original}]
To establish the base case, we exhibit Eulerian circuits that are 4-locally self-avoiding for each of the antiprisms except the octahedron. We must also do this for each graph obtained from the octahedron by one expansion operation. To this end, note that the octahedron does not have the local configuration required for pegging. Any 3-cycle slide results in a 3-connected quartic planar graph on 8 vertices, but the 4-antiprism is the unique such graph. Finally, 4-cycle addition results in a 3-connected quartic planar graph on 10 vertices. There are three isomorphism classes of such graphs of which one is the 5-antiprism, and another can be obtained from the 4-antiprism by pegging twice. Consequently, it turns out that there is only one additional graph to check. This together with the antiprisms are shown in Figure~\ref{basecase}, and have the following 4-locally self-avoiding Eulerian circuits with vertex labelling given in the figure;
\begin{itemize}
\itemsep=0mm
\item the 4-antiprism admits the circuit $abcy_2y_1x_1x_2x_3cax_1by_1x_2y_2x_3a$,
\item the 5-antiprism admits $abcy_3y_2y_1x_1x_2x_3x_4cax_1by_1x_2y_2x_3y_3x_4a$,
\item following the previous two particular cases, for a general $n$-antiprism we may take $abcy_{n-2}\ldots y_2y_1x_1x_2 \ldots x_{n-1}cax_1by_1x_2y_2x_3 \ldots x_{n-2}y_{n-2}x_{n-1}a$,
\item the graph obtained from the octahedron by applying a single 4-cycle addition has to good Eulerian circuit $abwevyxcbdafyexwvdfca$.
\end{itemize}

Given an arbitrary 3-connected quartic planar graph $G$, Theorem~\ref{hybridgen} implies the existence of a sequence of 3-connected quartic planar graphs $G_0,G_1\ldots G_n$ such that $G_n=G$, $G_0$ is one of the graphs mentioned in the base case, and $G_{i}$ can be constructed from $G_{i-1}$ by applying either a pegging operation, a 4-cycle addition or a 3-cycle slide for all $1\leq i \leq n$. We have verified above that $G_0$ has a 4-locally self-avoiding Eulerian circuit. Lemmas~\ref{cycleaddition}, \ref{cycleslide} and \ref{pegging} imply that each $G_i$ also admits such a circuit. In particular, this includes $G$.
\end{proof}

\section{Relaxing connectedness conditions}\label{sec:weaken}
We begin with a general lemma which will be used frequently as we relax the 3-connectedness condition in Theorem~\ref{thm:no34_original} over the subsequent sections.
\begin{lemma}\label{lengthlemma}
Let $G$ be a planar graph that is quartic except for four pairwise non-adjacent pendant vertices ($x_1$, $x_2$, $y_1$, and $y_2$), with $x_1a_1, x_2a_2, y_1b_1,y_2b_2 \in E(G)$, $a_1\neq a_2$, and $b_1\neq b_2$. In addition, suppose that there exist two edge-disjoint trails in $G$ that are each 4-locally self-avoiding and together cover $E(G)$. Then there exists another pair of trails with those properties and in addition either
\vspace{-3mm}
\begin{enumerate}
\itemsep=-1mm
\item one is an $(x_1,x_2)$-trail, the other is a $(y_1,y_2)$-trail, and both have length at least 5, or 
\item both are $(x_i,y_j)$-trails ($i,j=1,2$) with length at least 3.
\end{enumerate}
\end{lemma}
\begin{proof}
Given the trails in the hypotheses, suppose firstly that they are both $(x_i,y_j)$-trails ($i,j=1,2$). We may write them as $T_1=x_1Py_1$ and $T_2=x_2Qy_2$ after possibly relabelling $y_1$ and $y_2$. Neither trail has length 1 since the pendant vertices are non-adjacent, and they cannot both have length 2 since $a_1\neq a_2$. Thus, if they do not already satisfy the conclusion of the lemma, we may assume without loss of generality that $T_1$ has length 2 and $T_2$ has length at least 3. Let $z=a_1=b_1$ be the vertex in $P$, and split $Q=Q'zQ''$. As $a_2$ and $b_2$ are both distinct from $z$, it follows that $Q'$ and $Q''$ both have length at least 2, so the trails $x_1zQ''y_2$ and $x_2Q'zy_1$ satisfy the stipulated properties.

The other possibility is that we have trails of the form $x_1Px_2$ and $y_1Qy_2$. Given the preceding argument, it suffices to extract two edge-disjoint 4-locally self-avoiding $(x_i,y_j)$-trails that cover $E(G)$. If they do not already satisfy the conclusions of the statement, then without loss of generality we may assume that $x_1Px_2$ has length 3 or 4. In the former case, let $P=a_1a_2$. By possibly relabelling $y_1$ and $y_2$, we may assume that $a_1$ appears before $a_2$ in $Q$ (oriented from $y_1$ to $y_2$), then split $Q=Q'a_1Q''$ and take as our news trails $x_1a_1Q''y_2$ and $x_2a_2a_1{Q'}^{-1}y_1$. 

If $x_1Px_2$ has length 4, then write it as $x_1a_1ca_2x_2$. By possibly relabelling $y_1$ and $y_2$, we may assume that $a_1$ occurs before $a_2$ on $Q$, and split $Q=Q'a_1Q''a_2Q'''$. If $c\in Q'''$, then take the new trails $x_1a_1Q''a_2Q'''y_2$ and $x_2a_2ca_1{Q'}^{-1}y_1$. The situation is symmetric if $c\in Q'$, in which case we take the new trails $x_1a_1ca_2Q'''y_2$ and $x_2a_2{Q''}^{-1}a_1{Q'}^{-1}y_1$. If $c\in Q''$, then either of these pairs will do.
\end{proof}

We typically pair Lemma~\ref{lengthlemma} with the following construction. Let $H$ be a plane graph that is quartic except for four non-adjacent pendant vertices which lie on the same face $F$. Then $H$ can be \emph{completed into a quartic plane graph} by adding a vertex $z$ in $F$ adjacent to each of the four pendant vertices. There is a unique way to add four edges so that $z$ is surrounded by triangular faces. The resulting graph $H'$ is still simple, quartic and planar. This is useful because our theorems that guarantee good Eulerian circuits require these as a minimum. Given a graph $H$ with the required degree sequence and (non-)adjacencies, we can complete it into a quartic planar graph, apply some theorem or induction hypothesis to find a good Eulerian circuit, and then reverse the construction to obtain a pair of good trails satisfying the hypotheses of Lemma~\ref{lengthlemma}.

\begin{figure}[ht]
	\footnotesize
	\centering
	\unitlength=1cm
	 \scalebox{1}{
\begin{picture}(12.7,2.2)(0,0)
  \put(0,0){\includegraphics[scale=0.5]{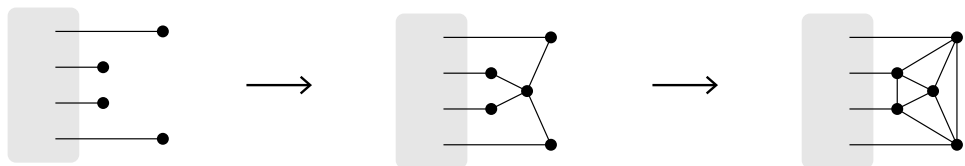}}
\end{picture}}
\caption{Completing $H$ into a quartic planar graph.}
\label{fig:3ec}
\end{figure}

One way to obtain pendant vertices from a planar graph is to \emph{split a vertex of degree 2}. That is, if we have a vertex $x$ in $G$ with $N(x) = \{u,v\}$, then the graph obtained by splitting $x$ is given by taking $G-x$ and adding two new vertices $x_1$ and $x_2$ together with the edges $ux_1$ and $vx_1$. These new edges replace $ux$ and $vx$ in the cyclic ordering at $u$ and $v$ respectively. More generally, we can split vertices of degree $d$ into $d$ pendant vertices. The inverse operation is given by \emph{identifying} pendant vertices.

\subsection{Down to 3-edge-connected}\label{sec:3edgecon}

\begin{theorem}\label{thm:no34edgecon}
Every 3-edge-connected quartic planar graph except the octahedron has a 4-locally self-avoiding Eulerian circuit.
\end{theorem}
\begin{proof}
We use induction on the number of 2-vertex-cuts, which in the 3-connected setting are all of type (a) (see Figure~\ref{fig:allcuts}). The base case is given by Theorem~\ref{thm:no34_original}. Let $G$ be a 3-edge-connected quartic planar graph with at least one 2-vertex-cut $\{x,y\}$, and fix a planar embedding. Denote the sides of the cut by $A$ and $B$. We shall construct an auxiliary graph from each side on which to apply the induction hypothesis. Let us proceed for now just with $A$. Split $x$ into $x_1$ and $x_2$, and $y$ into $y_1$ and $y_2$. The resulting graph can be completed into a quartic planar graph by the construction described earlier. Denote this graph by $A'$ with new vertex $z$, and let $K:= N(z) \cup\{z\}$. We claim that $A'$ is 3-edge-connected and has fewer 2-vertex-cuts than $G$.

We first show that $A'$ is 3-edge-connected. Since $G$ is 3-edge-connected, it follows that $A'-K$ is connected, so deleting any two edges with ends in $K$ cannot disconnect $A'$. Thus, if $\{e,f\} \subset E(A')$ is a 2-edge-cut, then at most one of these edges has an end in $K$. Suppose such a cut were to exist. Then as $K$ is contained in one component of $A'-\{e,f\}$, this would mean that $\{e,f\}$ is also an edge-cut in $G$, giving a contradiction.

It is similar to show that $A'$ has fewer 2-vertex-cuts than $G$. Let $\{u,v\}$ be a 2-vertex-cut in $A'$. Again since $A'-K$ is connected, it follows that $u$ and $v$ cannot both be in $K$. If $u$ is in $K$ and $v$ is in $V(A) - \{x,y\}$, then $v$ would be a cutvertex in $A'-u$. However, the remaining vertices of $K$ are all in the same component of this cut, so $v$ must be a cutvertex in $G$ which is also not possible. From here, it suffices to show that whenever $G-\{u,v\}$ is connected, then $A'-\{u,v\}$ is also connected for $u,v\in V(A) - \{x,y\}$, noting that $\{x,y\}$ is no longer present. There are certainly paths between any two vertices in $K$. Since $G-\{u,v\}$ is connected, there is a path $P_{st}$ between any two vertices $s$ and $t$ in $A$. If $P_{st}$ avoids $x$ and $y$, then it is also a path in $A'-\{u,v\}$. If it does traverse $x$, then it must have a subpath $xPy$, and to obtain a path in $A'$ we can replace this subpath with one in $K$ of the form $x_izy_j$. If exactly one of $s$ and $t$ are in $K$, say $t$, concatenate an $(s,x_i)$-path obtained from $P_{sx}$ with an appropriate $(x_i,t)$-path in $K$. 

Since $G$ is simple, $A-\{x,y\}$ cannot be a single vertex, so $A'$ is not isomorphic to the octahedron. Thus, by the induction hypothesis, there is a good Eulerian circuit in $A'$. Deleting $z$ as well as all edges between $x_1$, $x_2$, $y_1$ and $y_2$ leaves two edge-disjoint trails that together cover all edges of the remaining graph and are individually 4-locally self-avoiding. In particular, Lemma~\ref{lengthlemma} applies, so we may assume that we either have a good $(x_1,x_2)$-trail and a good $(y_1,y_2)$-trail both with length at least 5, or two $(x_i,y_j)$-trails ($i,j=1,2$) with length at least 3. Identifying $x_1$ and $x_2$ back into $x$, and $y_1$ and $y_2$ into $y$ means that we either have two good edge-disjoint circuits, one containing $x$ and the other containing $y$, that together cover all of the edges in $A$, or two good edge-disjoint $(x,y)$-trails that again cover all of the edges in $A$. Applying the same construction to $B$, we have in total four good trails that together cover $G$. To sew these together, we break into cases depending on which of these are closed.

\begin{enumerate}[label=(\alph*), topsep=0pt]
\item Suppose that $A$ and $B$ both have $(x,y)$-trails as described above. Let's call the trails $xPy, xQy\subset A$ and $yRx,ySx \subset B$. Then $xPyRxQySx$ is an Eulerian circuit in $G$. Since all the trails have length at least 3, the subcircuits $xPyRx$, $xQySx$, $yRxQy$ and $ySxPy$ all have length at least six so there are no short subcycles.

\item Suppose that one side, say $A$, has two $(x,y)$-trails $xPy$ and $yQx$, whilst the other, $B$, has two circuits $xRx$ and $ySy$ of length at least 5. Then $xRxPySyQx$ is a good Eulerian circuit in $G$.

\item Suppose that $A$ and $B$ each have two circuits, say $xPx$ and $yQy$ in $A$ and $xRx$ and $ySy$ in $B$. If there is some $v \in N(x) \cap N(y)$, assume without loss of generality that $v\in A$, and then write those trails as $xP'vx$ and $yQ'vy$, reversing $P$ or $Q$ if needed. Now $xP'vySyQ'vxRx$ is a good circuit in $G$. 

Otherwise, if $N(x) \cap N(y)$ is empty, then we will use a different auxiliary graph. Identify $x$ and $y$ in $A$ into a single vertex, say $u$, to obtain a graph $A''$. One can check that $A''$ is a 3-edge-connected simple quartic planar graph with fewer 2-vertex-cuts than $G$ using a very similar argument to what was used for the previous construction, so we leave out those details. In the same manner, construct $B''$ from $B$. Let's assume for now that they are not isomorphic to the octahedron. Then by the induction hypothesis, each of these has a good Eulerian circuit. If we split the identified vertex $u$ into four pendant vertices, then the conditions of Lemma~\ref{lengthlemma} are satisfied and, reversing the construction, we again have two trails in each of $A$ and $B$. If at least one of them has two $(x,y)$-trails, then it is handled by case (a) or (b). If both sides again have closed trails, then the good Eulerian circuit in $A''$ induces a good Eulerian circuit in $A$ (the vertices traversed directly before and after $x=y$ are either both neighbours of $x$, or both neighbours of $y$) and can be written as $xPx$, and similarly $xRx$ in $B$. These are easily combined into a good Eulerian circuit $xPxRx$ in $G$. If $A''$ is isomorphic to the octahedron, then $A$ must be isomorphic to the graph $H$ shown in Figure~\ref{fig:3ec} and we can use the explicit $(x,y)$-trails drawn in red and blue.
\end{enumerate}
\vspace{-1em}
\end{proof}

\begin{figure}[ht]
	\footnotesize
	\centering
	\unitlength=1cm
	 \scalebox{1}{
\begin{picture}(10,5)(0,0)
  \put(0.2,3.5){\includegraphics[scale=0.5]{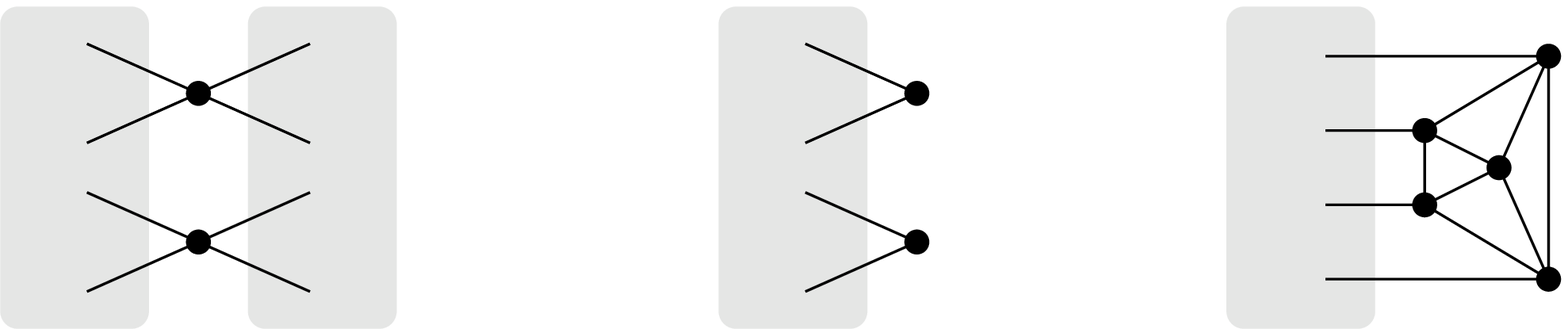}}
  \put(0.2,0.4){\includegraphics[scale=0.5]{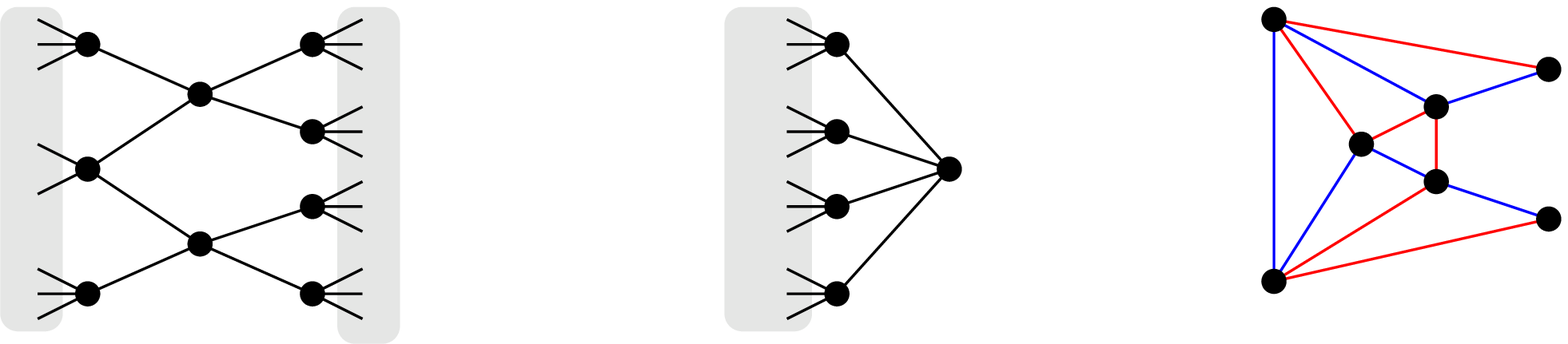}}
  \put(1.3,3){$G$}
  \put(1.4,5.2){$x$}
  \put(1.4,3.7){$y$}
  \put(5.3,3){$A$}
  \put(6.25,4.9){$x$}
  \put(6.25,4){$y$}
  \put(9,3){$A'$}
  \put(9.1,4.95){$x_1$}
  \put(9.1,4){$y_1$}
  \put(9.6,4.7){$z$}
  \put(10.3,5.2){$x_2$}
  \put(10.3,3.8){$y_2$}
  \put(1.3,-0.1){$G$}
  \put(1.4,2.2){$x$}
  \put(1.4,0.6){$y$}
  \put(0.7,1.7){$v$}
  \put(5.3,-0.1){$A''$}
  \put(6.4,1.45){$u$}
  \put(9,-0.1){$H$}
  \put(10.2,2.1){$x$}
  \put(10.2,1.1){$y$}
\end{picture}}
\caption{Construction of auxiliary graphs in the proof of Theorem~\ref{thm:no34edgecon}.}
\label{fig:3ec}
\end{figure}

\subsection{Further down to 2-connected}\label{sec:2con}
\begin{theorem}\label{thm:2con}
Let $G$ be a 2-connected quartic planar graph. Then $G$ has a 4-locally self-avoiding Eulerian circuit if and only if it does not contain $F_6$ (Figure~\ref{fig:obstructions}) as a subgraph.
\end{theorem}
\begin{proof}
It is routine to check that any Eulerian trail through $F_6$ must have a subcycle of length at most 4. 

Conversely, suppose that $G$ does not contain $F_6$ as a subgraph, or is $F_6$-free. We will use induction on the number of $2$-edge-cuts, with the base case given by Theorem~\ref{thm:no34edgecon}. For any 2-edge-cut $\{xs, yt\}$ in $G$ with sides $A \ni x,y$ and $B \ni s,t$, any Eulerian circuit in $G$ can be written in the form $xPytQsx$ where $xPy$ is an Eulerian trail in $A$ and $tQs$ is an Eulerian trail in $B$. Moreover, given good Eulerian trails in $A$ and $B$, one can combine them using the form above into a good Eulerian circuit in $G$. Thus, it is enough to work with just one side of the cut at a time, say $A$. 

We proceed much as we did in the proof of Theorem~\ref{thm:no34edgecon}. Again, we will construct an auxiliary graph $A'$ from $A$ that satisfies the conditions of the induction hypothesis, but now there is the added complication of ensuring that $A'$ is $F_6$-free. To do this, we define a class of graphs $\mathcal{C}$ such that $A'$ can only contain $F_6$ if $A$ contains a maximal 3-edge-connected subgraph isomorphic to one of the graphs in $\mathcal{C}$, a \emph{$\mathcal{C}$-subgraph}. Explicitly, $\mathcal{C}$ consists of $F_6-e$ for all edges $e$, $G_7$, $G_7-e$ for each edge $e\neq xy$, and $K_4$ (Figure~\ref{fig:nearmiss}). We say that a 2-edge-cut is incident to a \emph{$\mathcal{C}$-subgraph} if one end of each edge is a vertex of the subgraph.

\begin{figure}[ht]
	\centering
	\footnotesize
	\hspace{1mm}
	\unitlength=1cm
	\scalebox{1}{
\begin{picture}(15,3.2)(0,0)
  \put(0,0.7){\includegraphics[scale=0.5]{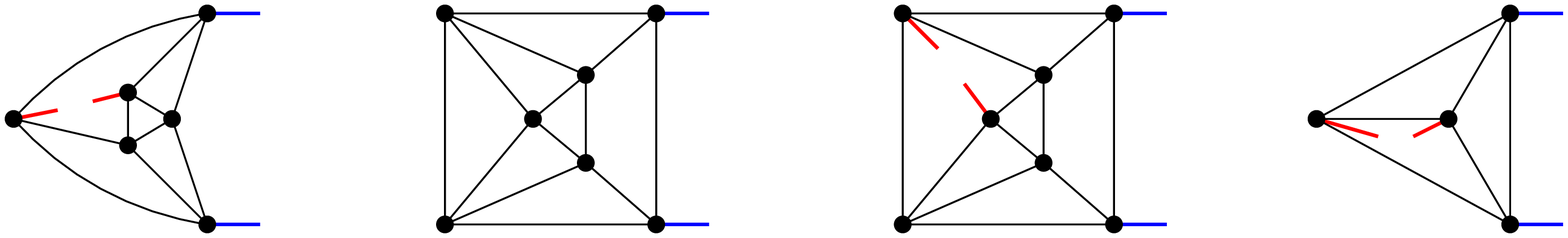}}
  \put(0.4,0){$F_6-ab$}
  \put(1.75, 2.85){$x$}
  \put(1.75, 0.45){$y$}
  \put(-0.25, 1.65){$a$}
  \put(1.7, 1.65){$c$}
  \put(1, 2.1){$b$}
  \put(1, 1.1){$d$}
  \put(4.9,0){$G_7$}
  \put(5.8, 2.85){$x$}
  \put(5.8, 0.45){$y$}
  \put(3.9, 2.85){$a$}
  \put(3.9, 0.4){$b$}
  \put(4.45, 1.65){$c$}
  \put(5.4, 2.0){$d$}
  \put(5.4, 1.3){$e$}
  \put(8.4,0){$G_7-ac$}
  \put(9.9, 2.85){$x$}
  \put(9.9, 0.45){$y$}
  \put(8, 2.85){$a$}
  \put(8, 0.4){$b$}
  \put(8.55, 1.65){$c$}
  \put(9.5, 2.0){$d$}
  \put(9.5, 1.3){$e$}
  \put(12.5,0){$K_4$}
  \put(13.5, 2.85){$x$}
  \put(13.5, 0.45){$y$}
  \put(11.5, 1.65){$v$}
  \put(13.2, 1.65){$w$}
\end{picture}}
\caption{Some graphs in $\mathcal{C}$ with incident 2-edge-cuts drawn as coloured half-edges.}
\label{fig:nearmiss}
\end{figure}

First, assume that $G$ has a $2$-edge-cut which is not incident to any $\mathcal{C}$-subgraph of $G$. Choosing such an edge-cut together with a side $A$, we break into four cases based on the local configuration around the cut as shown in Figure~\ref{fig:constructions_2vc}. In each case, we will find a good Eulerian trail through $A$.

\begin{figure}[ht]
	\centering
	\footnotesize
	\hspace{1mm}
	\unitlength=1cm
	\scalebox{1}{
\begin{picture}(13,6.4)(0,0)
  \put(0,0.7){\includegraphics[scale=0.5]{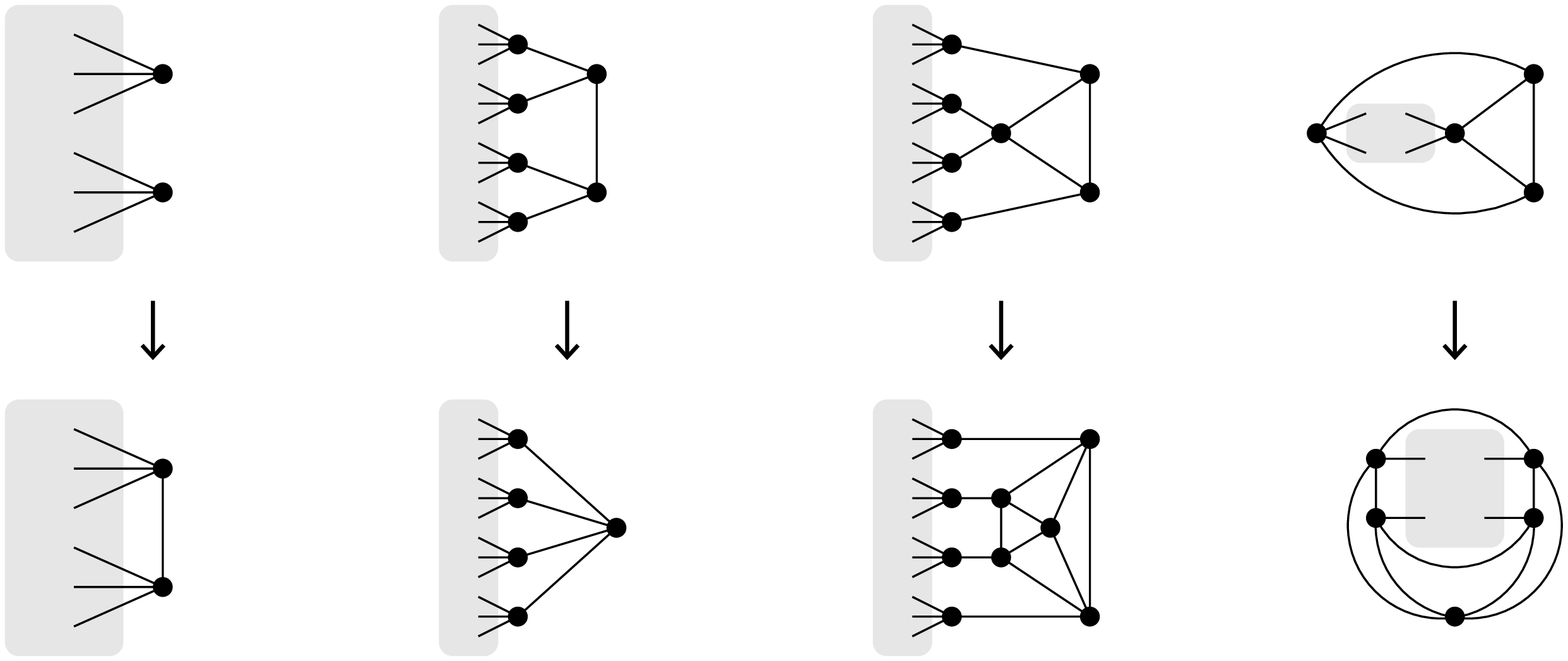}}
  \put(0.7,0){(a)}
  \put(1.2,5.55){$x$}
  \put(1.2,4.1){$y$}
  \put(1.45,2.15){$x$}
  \put(1.45,1.2){$y$}
  \put(4,0){(b)}
  \put(4.7,5.55){$x$}
  \put(4.7,4.1){$y$}
    {\scriptsize
  \put(4.2,5.75){$a_1$}
  \put(4.3,5.05){$a_2$}
  \put(4.3,4.65){$b_1$}
  \put(4.2,3.9){$b_2$}
  \put(4.05,2.65){$a_1$}
  \put(4.05,2.1){$a_2$}
  \put(4.03,1.2){$b_1$}
  \put(4.05,0.6){$b_2$}
  }
  \put(5.05,1.65){$v$}
  \put(7.7,0){(c)}
  \put(7.95,4.6){$v$}
  \put(8.65,5.55){$x$}
  \put(8.65,4.1){$y$}
  \put(8.65,2.6){$x$}
  \put(8.65,0.7){$y$}
  \put(7.85,2.15){$v_1$}
  \put(7.85,1.2){$v_2$}
   {\scriptsize
  \put(7.5,5.75){$b_1$}
  \put(7.5,5.3){$a_1$}
  \put(7.5,4.4){$a_2$}
  \put(7.5,3.9){$b_2$}
    \put(7.5,2.65){$b_1$}
  \put(7.5,2.15){$a_1$}
  \put(7.5,1.2){$a_2$}
  \put(7.5,0.7){$b_2$}
  \put(8.55,1.7){$c$}
  }
  \put(11.45,0){(d)}
  \put(12.2,5.55){$x$}
  \put(12.2,4.1){$y$}
  \put(11.6,5.1){$w$}
  \put(10.15,4.85){$v$}
  \put(10.55,2.35){$v_1$}
  \put(10.35,1.7){$v_2$}
  \put(12.45,2.35){$w_1$}
  \put(12.6,1.7){$w_2$}
\end{picture}}
\caption{The four possible sides of a 2-edge-cut and their auxiliary graphs.}
\label{fig:constructions_2vc}
\end{figure}

\begin{enumerate}[label=(\alph*)]
\item Suppose $x$ and $y$ are not adjacent. Then the graph $A'$ obtained by adding the edge $xy$ is a simple quartic planar graph. By considering the inverse construction, it is straightforward to show that $A'$ only contains $F_6$ if the chosen 2-edge-cut has two ends in a copy of $F_6-e \subset A$, so by our choice of cut $A'$ is $F_6$-free. Among a set of internally vertex-independent paths in $G$ between two vertices in $A$, at most one can traverse an edge of our chosen 2-edge-cut, and replacing an appropriate subpath (in $B$) of this path with $xy$ gives internally vertex-independent paths in $A'$. This, together with Menger's theorem, implies that $A'$ is 2-connected. We omit the details of showing that $A'$ has fewer 2-edge-cuts than $G$, but an argument similar to that in the proof of Theorem~\ref{thm:no34edgecon} suffices. Then by the induction hypothesis, there is a good Eulerian circuit in $A'$ which we can write as $xPyx$, and $xPy$ is a good Eulerian trail in $A$. 

\item Suppose $x$ and $y$ are adjacent and have no common neighbours, so $N(x) = \{a_1,a_2,y\}$ and $N(y) = \{b_1,b_2,x\}$ where all six labelled vertices are distinct. To construct $A'$ from $A$, contract $xy$ and call the new vertex $v$. As in the previous case, $A'$ is simple, quartic, planar and one can check in the same manner that it is 2-connected and has fewer 2-edge-cuts than $G$. In addition, $A'$ could only contain a copy of $F_6$ if our cut is incident to $G_7\in \mathcal{C}$ or $G_7-e\in\mathcal{C}$ which is not the case. Thus, by the induction hypothesis there is a good Eulerian circuit in $A'$. Such a circuit passes through $v$ twice, and hence contains either $a_1va_2$ and $b_1vb_2$ as subtrails, or $a_1vb_1$ and $a_2vb_2$ as subtrails after possibly relabelling $b_1$ and $b_2$. In the former case, we can write the circuit as $va_2Pb_1vb_2Qa_1v$. Then $xa_2Pb_1yb_2Qa_1xy$ is a good Eulerian trail in $A$. For the latter, the circuit in $A'$ has form either $va_1Pb_2va_2Qb_1v$ or $va_1Pa_2vb_2Qb_1v$, and we have good Eulerian trails $xa_1Pb_2yxa_2Qb_1y$ or $xa_1Pa_2xyb_2Qb_1y$ in $A$ respectively. There are no short subcycles since we have effectively just inserted an edge (namely, $xy$) into an already good circuit.

\item Suppose $x$ and $y$ are adjacent and have exactly one neighbour in common, say $v$. Let $N(v) = \{x,y,a_1,a_2\}$, $N(x) = \{s,v,y, b_1\}$ and $N(y) = \{s,v,x, b_2\}$. We know that $a_1 \neq a_2$, $b_i \neq v$ ($i=1,2$) since $A$ is simple, and $b_1 \neq b_2$ since otherwise $x$ and $y$ would have two neighbours in common. Then splitting the vertex $v$ into $v_1$ and $v_2$, we can complete this into a quartic planar graph $A'$ with new vertex $c$ which is 2-connected and has fewer 2-edge-cuts than $G$ by the same reasoning as in the previous cases. By noting that $N(x) \cap N(y) = N(v_1)\cap N(v_2)= \{c\}$, these vertices cannot all be in one copy of $F_6$ and it follows that $A'$ must be $F_6$-free. Thus, we may apply the induction hypothesis and then delete the vertices and edges added in the completion to obtain two Eulerian trails subject to Lemma~\ref{lengthlemma}. These are either of the form $v_1Pv_2$ and $xQy$ (with $|P|, |Q| \geq 5$), or else of the form $xPv_1$ and $yQv_2$ (with $|P|, |Q| \geq 3$) after possibly relabelling $v_1$ and $v_2$. In the first case, one of $xQyxvPvy$ and $xQyxvP^{-1}vy$ will be a good Eulerian trail in $A$, chosen so that $p_1$ is not $q_{-1} = b_2$. In the second case, the trail $xPvxyQvy$ works unless $p_{-1} = q_1 = b_2$. We can then let $P = P'b_2$ and $Q=b_2Q'$, so that $xP'b_2yxv{Q'}^{-1}b_2vy$ is a good trail in $A$.

\item The only remaining possibility is that $x$ and $y$ are adjacent and have two neighbours in common, say $v$ and $w$. We may assume that $vw \not\in E(G)$, since otherwise $A$ would contain a copy of $K_4$ incident to the chosen cut. Let $N(v) = \{x,y,a_1,a_2\}$ and $N(w) = \{x,y,b_1,b_2\}$, noting that $a_1\neq a_2$ and $b_1\neq b_2$ since $G$ is simple. Splitting $v$ and $w$ creates four pendant vertices, and we can complete this into a quartic planar graph $A'$. The induction hypothesis and Lemma~\ref{lengthlemma} apply in the usual manner, so we have two Eulerian trails, either $v_1Pv_2$ and $w_1Qw_2$ ($|P|, |Q| \geq 5$), or $v_1Pw_1$ and $v_2Qw_2$ ($|P|, |Q| \geq 3$). Then $xvPvyxwQwy$ and $xvPwxyvQwy$ are good Eulerian trails in $A$ respectively.
\end{enumerate}

It remains to handle graphs in which every 2-edge-cut is adjacent to a $\mathcal{C}$-subgraph. We again only work with one side at a time, say $A_0$ with vertices $x$ and $y$, and assume that $A_0$ contains a $\mathcal{C}$-subgraph $H_0$ incident to the chosen cut. If $H_0\cong G_7$, then $xacbedxybadcey$ is a good Eulerian trail. Otherwise, by inspection we observe that $H_0$ must actually be incident to another 2-edge-cut $C_1$ distinct from $C_0$. Let $A_1$ be the side of $C_1$ not containing $B_0$. If $A_1$ contains another $\mathcal{C}$-subgraph adjacent to $C_1$, say $H_1$, then it is incident to another cut $C_2$ from which we define $A_2$ to be the side not containing $B_0$. Continuing this way, by finiteness of $G$ there is some $i$ such that $C_{i}$ is incident to only one $\mathcal{C}$-subgraph, namely $H_{i-1}$. Then either $H_{i-1} \cong G_7$ in which case we have already given an explicit trail, or else a good trail can be obtained by constructing $A_i'$, which is guaranteed to be $F_6$-free, and applying the induction hypothesis using the appropriate case above.

To complete the proof, we provide a pair of good trails that cover the edges of each graph in $\mathcal{C}$ except $G_7$, as these can be stitched together to extend the good Eulerian trail in $A_i$ to all of $A_0$. We use the labelling from Figure~\ref{fig:nearmiss}.
\begin{itemize}
\itemsep=0mm
\item For $F_6-e$, if $e=xa, xb, eb$ then take the pair of trails obtained by splitting $xcbdyabxadcy$ at the missing edge, and if $e=bd$ take the trails $xbcdyab$ and $daxcy$. All other $e$ can be inferred by relabelling.
\item For $G_7-e$, split the trail $xacbedxybadcey$ at the missing edge.
\item For $K_4$, the trails $xvyw$ and $vwxy$ will do. 
\end{itemize}
\vspace{-2em}
\end{proof}

\subsection{Handling cutvertices}\label{sec:1con}

Recall the pegging operation introduced in Section~\ref{sec:3con}. There is a more general version of this operation in which we do not require $b$ and $c$ (labelled as in Figure~\ref{fig:expansions}) to be adjacent. We will call a 2-edge-cut \emph{independent} if it consists of independent edges. The edges of any independent 2-edge-cut in a planar graph have a common face, and are hence peggable in this general sense. It is easy to show that this special case of pegging, which in particular produces a cutvertex, preserves the property of having a 4-locally self-avoiding Eulerian circuit. 

\begin{lemma}[Special pegging]\label{lemma:unpegspecial}
Let $G$ be a connected quartic plane graph with an independent 2-edge-cut, say $\{ab, cd\}$. Let $G'$ be the graph obtained after pegging $ab$ and $cd$, and $x$ be the new vertex (Figure~\ref{fig:3ec_b}). If $G$ has a 4-locally self-avoiding Eulerian circuit, then so does $G'$. 
\end{lemma}
\begin{proof}
Any Eulerian circuit in $G$ must have the form $aPdcQba$. Letting this be a good circuit, we claim that $aPdxcQbxa$ is a good circuit in $G'$. Here, $aPd$ and $cQb$ are good Eulerian trails, so any short subcycle must involve $x$. The only possibilities are $xcQbx$ and $xaPdx$ which are short subcycles if $aPd$ or $cQb$ have fewer than 3 edges. To see that this cannot be the case, we note that $a\neq d$ and $b\neq c$ since $G$ was simple, and there must be a third vertex in each side of the edge-cut as $G$ is quartic. That gives at least four edges that must in the Eulerian trail $aPd$ in one side, and $cQb$ in the other.
\end{proof}

\begin{figure}[ht]
	\footnotesize
	\centering
	\unitlength=1cm
	 \scalebox{1}{
\begin{picture}(8,2)(0,0)
  \put(0.2,0.5){\includegraphics[scale=0.5]{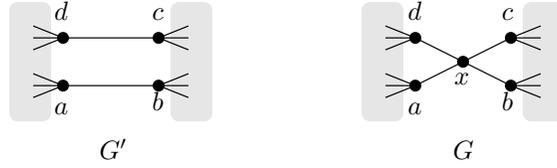}}
  \put(1.4,0){$G'$}
  \put(6.12,1){$x$}
  \put(0.8,1.85){$d$}
  \put(0.8,0.60){$a$}
  \put(2.1,1.85){$c$}
  \put(2.1,0.65){$b$}
  \put(6.1,0){$G$}
  \put(5.5,1.85){$d$}
  \put(5.5,0.60){$a$}
  \put(6.75,1.85){$c$}
  \put(6.75,0.65){$b$}
\end{picture}}
\caption{Pegging an independent 2-edge-cut produces a cutvertex.}
\label{fig:3ec_b}
\end{figure}

\begin{proof}[Proof of Theorem~\ref{thm:connected}]
It is routine to verify that every Eulerian trail through $F_6$ has a subcycle of length at most 4. For the converse, we use induction on the number of cutvertices. The base case is given by Theorem~\ref{thm:2con}. Suppose $G$ is planar and quartic with a cutvertex $x$. Let $N(x) = \{a_1,a_2,b_1,b_2\}$ with $a_1,a_2$ on one side of the cut and $b_1,b_2$ on the other. We construct an auxiliary graph $G'$ with $V(G') = V(G)-\{x\}$ and $E(G') = E(G-x)\cup\{a_1b_1,a_2b_2\}$ (labelled so that $G'$ is planar) given by unpegging at $x$. It is clear that $G'$ is connected, quartic and planar. It is $F_6$-free since $G$ was $F_6$-free and $F_6$ is 3-edge-connected whereas unpegging produces a 2-edge-cut, and simple because the new edges cannot be parallel to any edge with ends in the same component of $G-x$, nor can they be parallel to each other (this would mean $a_1x$ and $a_2x$ or $b_1x$ and $b_2x$ were parallel in $G$). Finally, we observe that $G-x$ has exactly two connected components for degree reasons, from which it follows that every cutvertex in $G'$ must be a cutvertex in $G$. With $x$ gone, $G'$ has fewer cutvertices than $G$. Thus, by the induction hypothesis, there exists a good Eulerian circuit in $G'$. Repegging the two new edges recovers $G$, so by Lemma~\ref{lemma:unpegspecial} we conclude that $G$ also admits a good circuit.
\end{proof}

\section{Cutting into paths}\label{sec:cutitup}

Given a graph $G$ together with a startvertex $v\in V(G)$ and a sequence $L=(\ell_1, \ell_2, \ldots, \ell_k)$ of natural numbers, a \emph{$P_{(L, v)}$-decomposition} of $G$ is a sequence of edge-disjoint paths such that the $i$th path $P_i$ has length $\ell_i$, the first vertex of $P_1$ and the last vertex of $P_k$ are both $v$, and the last vertex of $P_i$ is the first vertex of $P_{i+1}$ for all $1\leq i \leq k-1$. That is, the paths can be stitched end to end in sequence to obtain an Eulerian circuit. A sequence of lengths $L'$ is a \emph{subdivision} of a sequence of lengths $\underline{L}$ in which some lengths are underlined if $L'$ can be obtained from $\underline{L}$ by iteratively picking a length $\ell_i$ that is not underlined, and replacing it with two or more smaller lengths that sum to $\ell_i$. For example, $(1,2,4,4)$ and $(1,1,1,4,3,1)$ are subdivisions of $(3,\underline{4},4)$, but $(3,2,2,4)$ and $(2,2,4,3)$ are not.

By ``cutting up" a 4-locally self-avoiding Eulerian circuit, it is possible to obtain path decompositions with any startvertex and any sequence of prescribed lengths provided that the total length of the path is equal to the number of edges in our graph. Theorem~\ref{thm:connected} therefore gives us a way to decompose connected $F_6$-free quartic planar graph into paths of length at most 4. It turns out that this also extends to graphs that do contain copies of $F_6$.

\begin{lemma}\label{lem:f6path}
Let $L$ be any sequence $(\ell_1, \ell_2, \ldots, \ell_k)$ such that $\ell_i = 1,2,3,4$ for each $i$, and $\sum_i \ell_i = 11$. Then $F_6$, labelled as in Figure~\ref{fig:obstructions}, has a $P_{(L, x)}$-decomposition.
\end{lemma}
\begin{proof}
If $L$ is a subdivision of $(3,\underline{4},4)$, then the decomposition can be obtained by cutting up the Eulerian circuit $xadbxcyabcdy$. If not, then instead cut up the Eulerian circuit $xadbcyabxcdy$, noting that the only short subcycle is $bcyab$.
\end{proof}

\begin{cor}
Let $G$ be a connected quartic planar graph with a chosen startvertex $v$, and $L$ be a sequence $(\ell_1, \ell_2, \ldots, \ell_k)$ such that $\ell_i = 1,2,3,4$ for each $i$. Then $G$ has a $P_{(L, v)}$-decomposition if and only if $\sum_i \ell_i = |E(G)|$.
\end{cor}
\begin{proof}
Necessity is clear. For the converse, if $G$ is $F_6$-free, then by Theorem~\ref{thm:connected} there exists a 4-locally self-avoiding Eulerian circuit which, starting from the specified vertex, can be cut up into the required paths.

In the following, we use the labelling of $F_6$ and the octahedron from Figure~\ref{fig:obstructions}. Suppose $G$ isomorphic to the octahedron. Then $G$ is vertex-transitive, so without loss of generality we may suppose that the startvertex is $x$. If $L$ is a subdivision of $(3, \underline{4}, 5)$ or $(5, \underline{4}, 3)$ then the required decomposition can be obtained by cutting the Eulerian circuits $xadcbxydbaycx$ or $xcyabdyxbcdax$ respectively. For all other $L$, we can use the circuit $xadbcyxbaydcx$.

All remaining graphs $G$ contain at least one copy of $F_6$ as an induced subgraph. Construct a graph $G'$ from $G$ by replacing each copy of $F_6$ with a planar graph that is $F_6$-free, quartic except for two vertices of degree 3, and of order greater than 6. By Theorem~\ref{thm:connected}, $G'$ has 4-locally self-avoiding Eulerian circuit, say $C'$. The idea is to obtain the required decomposition by following $C'$ for as long as it traverses edges common to $G'$ and $G$, which is to say until we reach a copy of $F_6$ in $G$, and cutting as we go. The preceding lemma then supplies an appropriate path-decomposition through $F_6$, before continuing back on $C'$. The remainder of the proof describes this more precisely.

Let's first handle the situation when startvertex $v\in V(G)$ is not degree 4 in any copy of $F_6$ so that it can be naturally identified with a vertex of $G'$ and hence $C'$. Then starting at $v$, we begin to partition $C'$ into paths with lengths according to $L$ until we reach a vertex, say $x$, that is contained in a copy of $F_6$. Assume also that the last complete path in the decomposition was $P_{i-1}$, and that at this point we already have the first $n$ edges of $P_{i}$ where $0 \leq n < \ell_i$. Let $L_F = (\ell_i-n, \ell_{i+1}, \ldots, \ell_j -m)$, where $0\leq m < \ell_j$, be the uniquely determined length vector that represents the next 11 edges of $L$. By Lemma~\ref{lem:f6path}, $F_6$ admits a $P_{(L_F, x)}$-decomposition. We now add all of these paths to our decomposition sequence, after concatenating the first path with those $n$ edges of $P_i$ that were already determined. This process takes us through all of the edges in a copy of $F_6$, so the next edge (the only one adjacent to $y$ in $G$ that is not yet in a path) is also an edge of $C'$. Thus, from $y$ we can again proceed to follow $C'$, with the remaining length sequence being $(m, \ell_{j+1}, \ldots, \ell_k)$.

If $v$ does have degree 4 in a copy of $F_6$, then we may assume that $v=a$ or $v=d$. If $v=a$, then since $xa$ is the first edge in both Eulerian paths in the proof of Lemma~\ref{lem:f6path}, we replace $L$ by $L'=(1, \ell_1, \ldots, \ell_k-1)$ deleting the final length if it is 0, and pick a new startvertex $v'=x$. Now we can find a $P_{(L', v')}$-decomposition of $G$. By possibly joining the first path to the end of the last according to whether $\ell_k=1$ or not, we then obtain a $P_{(L,v)}$-decomposition. Similarly, if $v=d$, then $dy$ is the last edge in both of those paths. Here, we let $L'$ be $(\ell_1-1, \ell_2, \ldots \ell_k, 1)$ deleting the first length if it is 0, and $v' = y$.
\end{proof}
As a special case, taking $L$ to be a sequence of all 4s proves Theorem~\ref{thm:p5decomp}.

\section{Further questions}
In terms of locally self-avoiding Eulerian circuits, Theorem~\ref{thm:connected} is the best possible result for $F_6$-free quartic planar graphs in the sense that graphs in this class may contain $G_7$ (Figure~\ref{fig:nearmiss}), and every Eulerian path through $G_7$ has a subcycle of length at most 5. However, there is already a gap if we return to 3-connected quartic planar graphs. In that case, the family of graphs obtained by gluing an odd number of copies of the graph in Figure~\ref{fig:bbgraphs}(a), communicated by Brendan McKay and Be\'ata Faller, all have the property that every Eulerian circuit contains a subcycle of length at most 8. We have shown that 4-locally self-avoiding Eulerian circuits exist with only one exception for this class, so it remains to pinpoint the greatest $k$ such that all but finitely many 3-connected quartic planar graphs admit a $k$-locally self-avoiding Eulerian circuit.

\begin{figure}[ht]
	\footnotesize
	\centering
	\unitlength=1cm
	 \scalebox{1}{
\begin{picture}(14.5,2.9)(0,0)
\put(0,0.8){\includegraphics[scale=0.25]{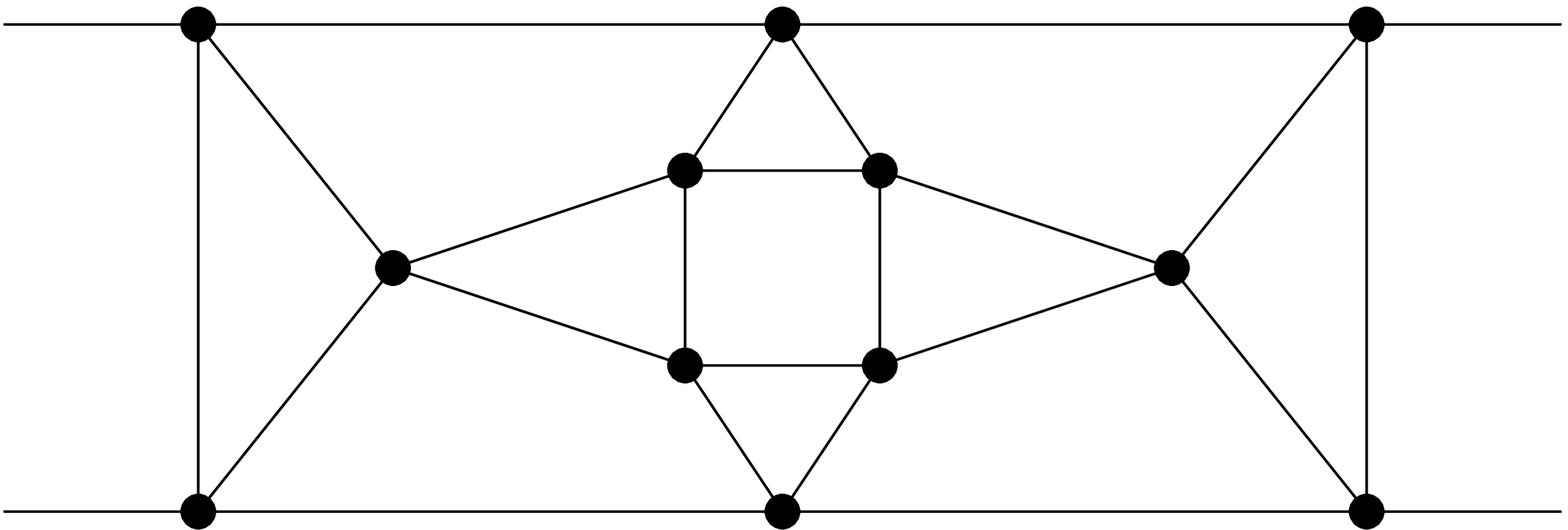}}
  \put(6.3,0.4){\includegraphics[scale=0.10]{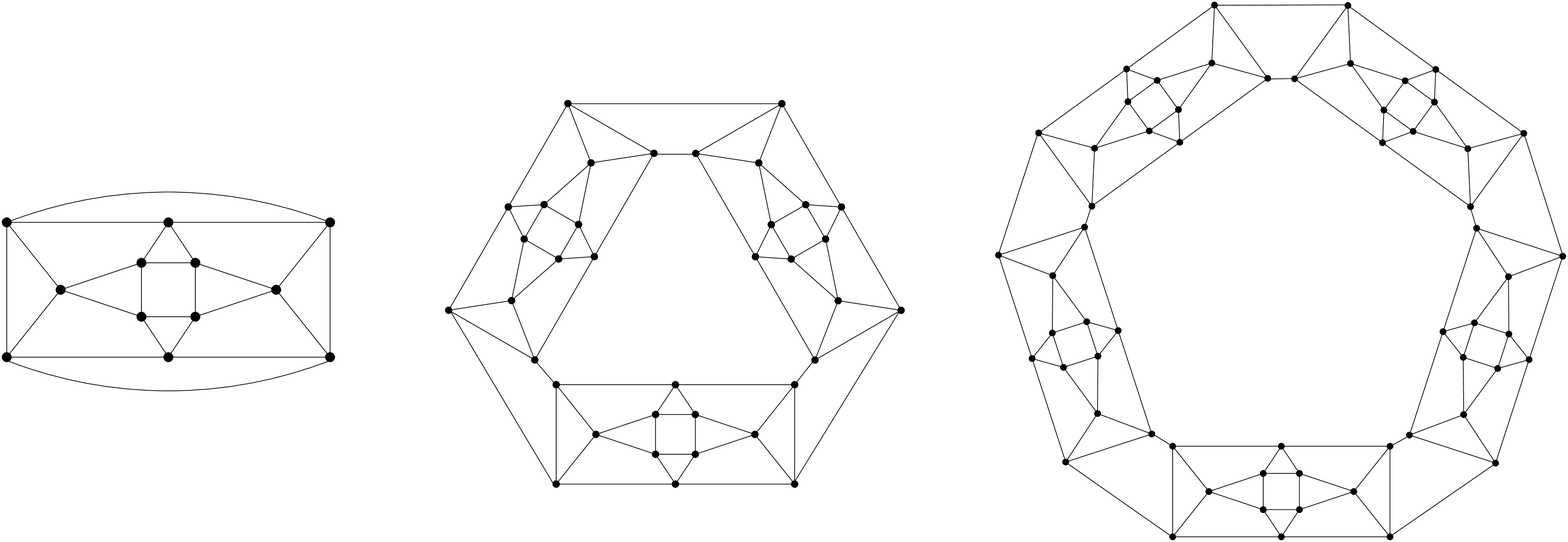}}
  \put(14.2,1.5){\normalsize{\textbf{$\cdots$}}}
  \put(2.35,0){(a)}
  \put(10.1,0){(b)}
\end{picture}}
\caption{A gadget (a) used to construct the family of graphs depicted in (b), in which every Eulerian circuit has a subcycle of length at most 8.}
\label{fig:bbgraphs}
\end{figure}

Another natural question is to ask for the largest $k$ such that for all but finitely many quartic planar graphs, possibly with some connectivity conditions, the trivial divisibility condition is also sufficient for admitting a $P_k$-decomposition. For connected quartic planar graphs, Theorem~\ref{thm:p5decomp} says that $k\geq 5$. We can also deduce that $k< 7$ since for any graph in the class that contains $F_6$, a $P_7$-decomposition would require a path on 7 vertices in $F_6$.

Finally, it would be nice to drop planarity from Theorem~\ref{thm:connected}. That is, which quartic graphs admit a 4-locally self-avoiding Eulerian circuit? In addition to $F_6$, the obstruction $K_5-e$ is known from the characterisations by Adelgren~\cite{Ade95} and Heinrich, Liu and Yu~\cite{HLY99} in the 3-locally self-avoiding case. There are already recursive generation theorems for this class (see \cite{DKS10, BJF83}), but to avoid having to prove an analogue of Lemma~\ref{pegging} without planarity, it may be necessary to use different expansion operations.

 \section*{Acknowledgements}
I wish to thank Brendan McKay, Scott Morrison, and Catherine Greenhill for their guidance, generous support, and helpful comments. I would also like to thank the University of New South Wales for hosting me during part of this work.

\bibliography{refs}
\bibliographystyle{abbrv}

\end{document}